\documentclass[11pt]{amsart}
\usepackage{amsmath,amssymb,amsthm,pinlabel,tikz,hyperref,mathrsfs,color, thmtools}
\usepackage{verbatim}
\usepackage{fullpage}
\usepackage{float}
\usepackage{caption}
\usepackage{subcaption}
\usepackage{enumitem}
\newcommand{\nc}{\newcommand}
\nc{\dmo}{\DeclareMathOperator}
\dmo{\ra}{\rightarrow}
\dmo{\Prob}{\mathbb{P}}
\dmo{\E}{\mathbb{E}}
\dmo{\N}{\mathbb{N}}
\dmo{\Z}{\mathbb{Z}}
\dmo{\Q}{\mathbb{Q}}
\dmo{\R}{\mathbb{R}}
\dmo{\C}{\mathcal{C}}
\dmo{\X}{\mathcal{X}}
\dmo{\U}{\mathcal{U}}
\dmo{\T}{\mathcal{T}}
\dmo{\F}{\mathcal{F}}
\dmo{\AC}{\mathcal{AC}}
\dmo{\AxN}{\mathrm{[id, \varphi^{N}]}}
\dmo{\AxCN}{\Upsilon_{N}}

\dmo{\w}{\omega}
\dmo{\MIN}{\mathcal{MIN}}
\dmo{\Mod}{Mod}
\dmo{\PMod}{PMod}
\dmo{\PMF}{\mathcal{PMF}}
\dmo{\Mat}{Mat}
\dmo{\supp}{supp}
\dmo{\UE}{\mathcal{UE}}
\dmo{\vol}{vol}
\dmo{\B}{B}
\dmo{\PB}{PB}
\dmo{\PR}{PSL(2,\mathbb{R})}
\dmo{\GL}{GL(k, \mathbb{C})}
\dmo{\SL}{SL(2, \mathbb{Z})}
\dmo{\Isom}{Isom}
\dmo{\RP}{\mathbb{R} \mathrm{P}}
\dmo{\I}{\mathcal{I}}
\dmo{\el}{\ell_{\C}}
\dmo{\NN}{\mathcal{N}}
\dmo{\rk}{rank}
\dmo{\tr}{tr}
\dmo{\llangle}{\langle\langle}
\dmo{\rrangle}{\rangle\rangle}
\dmo{\Unif}{Unif}
\dmo{\Out}{Out}
\dmo{\Aut}{Aut}
\dmo{\Proj}{\operatorname{Proj}}
\dmo{\sumRho}{\mathcal{N}}
\dmo{\stopping}{\vartheta}
\dmo{\diam}{\operatorname{diam}}

\usetikzlibrary{decorations.markings, patterns}
\tikzset{->-/.style={decoration={
  markings,
  mark=at position #1 with {\arrow{>}}},postaction={decorate}}}

\nc{\nt}{\newtheorem}

\nt{theorem}{Theorem}

\newtheorem{thm}{{\bf Theorem}}[section]
\newtheorem{lem}[thm]{{\bf Lemma}}
\newtheorem{cor}[thm]{{\bf Corollary}}
\newtheorem{prop}[thm]{{\bf Proposition}}
\newtheorem{fact}[thm]{Fact}

\newtheorem{claim}[thm]{Claim}

\newtheorem{question}[thm]{Question}

\newtheorem{dfn}[thm]{Definition}
\numberwithin{equation}{section}

\title[Counting WPD elements]{Acylindrically hyperbolic groups and counting problems}

\date{\today}
\author{Inhyeok Choi}

\address{%
		Cornell University\\
		310 Malott Hall, Ithaca, NY, USA \\ \newline
		June E Huh Center for Mathematical Challenges, KIAS\\
		85 Hoegi-ro, Dongdaemun-gu, Seoul 02455, South Korea
}
\email{
        inhyeokchoi48@gmail.com
        }

\begin{document}
\begin{abstract}
We show that Morse elements are generic in acylindrically hyperbolic groups. As an application, we observe that fully irreducible outer automorphisms are generic in the outer automorphism group of a finite-rank free group.


\noindent{\bf MSC classes:} 20F67, 30F60, 57K20, 57M60, 60G50
\end{abstract}

\maketitle

%
%

\section{Introduction}	\label{sec:introduction}

In non-positively curved manifolds and groups, certain geodesics or group elements exhibit  hyperbolicity. A quasi-geodesic $\gamma$ is said to be \emph{Morse} if every quasi-geodesic of uniform quality connecting points on $\gamma$ lies in a common neighborhood of $\gamma$. A group element $g$ is called a \emph{Morse element} if its orbit $\{g^{i}\}_{i \in \Z}$ is an unbounded Morse quasi-geodesic in the group.

In globally hyperbolic spaces such as  CAT($-1$) spaces and Gromov hyperbolic spaces, every geodesic is Morse (of uniform quality). This corresponds to the fact that every infinite-order element in a word hyperbolic group is loxodromic and is Morse. Furthermore, ``most'' elements in a  word hyperbolic group are Morse. To formulate this, given a group $G$ and its generating set $S$, let $B_{S}(n)$ be the collection of group elements whose $S$-word length is at most $n$. We can ask if the proportion of Morse elements in $B_{S}(n)$ tends to 1 as $n$ tends to infinity. This is indeed the case when $G$ is an infinite word hyperbolic group \cite{daniGenericity}, \cite{gekhtman2018counting}, \cite{yang2020genericity}.

Morse elements are found in many other groups with flat parts. One classic example is the mapping class group $\Mod(\Sigma)$ of a finite-type hyperbolic surface $\Sigma$, whose Morse elements are precisely pseudo-Anosov mapping classes. In \cite{choi2024counting}, the author proved that the asymptotic density of pseudo-Anosovs in the mapping class group is 1. We establish a similar result for the class of \emph{acylindrically hyperbolic groups}. Our main theorem is:

\begin{theorem}\label{thm:main}
Let $G$ be an acylindrically hyperbolic group. Then for any finite generating set $S$ of $G$, we have \[
\lim_{n\rightarrow +\infty} \frac{\# \{ g \in B_{S}(n) : \textrm{$g$ is Morse} \} }{\#B_{S}(n)} = 1.
\]
\end{theorem}
This generalizes W. Yang's result on groups with strongly contracting element \cite{yang2020genericity}. This can be also compared with A. Sisto's theorem that simple random walks on acylindrically hyperbolic groups favor Morse elements (\cite[Theorem 1.6]{sisto2018contracting}). In fact, non-elementary random walks on any Gromov hyperbolic space favor loxodromics \cite{calegari2015scl}, \cite{maher2018random}, but one cannot hope such a result for counting problems (see the following subsections).

In view of the equivalent definitions of acylindrically hyperbolic groups in \cite{osin2016acylindrically} (especially in relation to \cite{bestvina2002bounded}), Theorem \ref{thm:main} is a restatement of the following more explicit theorem.

\begin{thm}\label{thm:mainWPD}
Let $G$ be a group generated by a finite set $S \subseteq G$. Suppose that $G$ acts on a Gromov hyperbolic space $X$ and that there exists $g \in G$ that is a loxodromic isometry of $X$ with the WPD property (cf. Definition \ref{dfn:WPD}). Then for any $M>0$, we have \[
\lim_{n\rightarrow +\infty} \frac{\# \big\{ g \in B_{S}(n) : \textrm{$g$ is WPD loxodromic and satisfies $\tau_{X}(g) > M$} \big\} }{\#B_{S}(n)} = 1.
\]
\end{thm}

Indeed, if $g \in G$ serves as a WPD loxodromic on a Gromov hyperbolic space, then $g$ is a Morse element in $G$ (\cite[Theorem 1]{sisto2016quasi-convexity}, \cite[Theorem 1.4]{osin2016acylindrically}).

We note a theorem by B. Wiest \cite{wiest2017on-the-genericity} that was applied to the mapping class group by M. Cumplido and B. Wiest \cite{cumplido2018pA}: for any finitely generated group $G$ having a non-elementary action on a Gromov hyperbolic space, the density of loxodromics is bounded away from 0. Hence, the main point of Theorem \ref{thm:main} and \ref{thm:mainWPD} is that the density has limit 1. Such a claim does not hold for general non-elementary actions.

Two important examples of acylindrically hyperbolic groups beyond hierarchically hyperbolic groups (HHGs) are $\Out(F_{N})$ and $\Aut(F_{N})$, the outer automorphism group and the automorphism group of the free group of rank $N \ge 3$. Theorem \ref{thm:main} tells us that most elements are Morse in large word metric balls in these groups. 

We can say more by focusing on a specific $\Out(F_{N})$-action, namely, the one on the free factor complex $\mathcal{FF}_{N}$ studied by M. Bestvina and M. Feighn \cite{bestvina2014hyperbolicity}. Bestvina and Feighn proved that: \begin{enumerate}
\item $\mathcal{FF}_{N}$ is Gromov hyperbolic,
\item the elements of $\Out(F_{N})$ that are loxodromic isometries of $\mathcal{FF}_{N}$ are precisely the fully irreducible outer automorphisms, and 
\item every fully irreducible outer automorphism has the WPD property.
\end{enumerate}
We have:

\begin{thm}\label{thm:mainGen}
Let $G = \Out(F_{N})$ be the outer automorphism group of the free group of rank $N$ for some $N\ge 2$. Then for any finite generating set $S$ of $G$, we have \[
\lim_{n\rightarrow +\infty} \frac{\# \{ g \in B_{S}(n) : \textrm{$g$ is an ageometric triangular fully irreducible element} \} }{\#B_{S}(n)} = 1.
\]
\end{thm}

This is a counting version of Kapovich--Maher--Pfaff--Taylor's result that random walks on $\Out(F_{N})$ favor ageometric triangular fully irreducibles \cite[Theorem A]{kapovich2022random}. There are also versions of random walk theory on $\Aut(F_{N})$ and $\Out(F_{N})$ using ``non-backtracking'' paths (\cite{kaimanovich2007subadditive}, \cite{kapovich2015a-train}). In particular, I. Kapovich and C. Pfaff proved that non-backtracking random walks favor geometric triangular fully irreducibles as well.

We record a cute application to the mapping class group $\Mod(\Sigma)$. It seems hard to apply the method of \cite{choi2024counting} to general non-elementary subgroups of $\Mod(\Sigma)$. However, since they all act on the curve complex $\mathcal{C}(\Sigma)$ with a WPD loxodromic element, we observe that: 

\begin{cor}\label{cor:mainMCG}
Let $G \le \Mod(\Sigma)$ be a non-elementary subgroup of the mapping class group and let $S$ be a finite generating set of $G$. Then for any $M>0$, we have \[
\lim_{n \rightarrow +\infty} \frac{ \# \big\{ g \in B_{S}(n) : \textrm{$g$ is pseudo-Anosov with stretch factor $\ge M$}\big\}}{\#B_{S}(n)} = 1.
\]
\end{cor}
In particular, pseudo-Anosovs are generic in the Torelli group. This generalizes the result of I. Gekhtman, S. Taylor, and G. Tiozzo regarding word hyperbolic groups acting on a Gromov hyperbolic space \cite[Theorem 1.12]{gekhtman2018counting}.

\subsection{Comparison with other groups} \label{subsection:compare}

To better illustrate Theorem \ref{thm:mainWPD}, let us compare four groups that act on a Gromov hyperbolic space: the free group $F_{2}$ of rank 2, the mapping class group $\Mod(\Sigma)$, the outer automorphism group $\Out(F_{N})$ of the free group of rank $n \ge 2$, and the direct product $F_{2} \times F_{3}$ of two free groups. All of these act on some Gromov hyperbolic space.

Let $G$ be a group acting on a hyperbolic space $X$ and let $S$ be a finite generating set of $G$. 
Given a group element $g \in G$ and a basepoint $x_{0} \in X$, let us define a function $f: \mathbb{R}^{2} \rightarrow \mathbb{R} \cup \{+\infty\}$ as follows. For any $M$-short word metric geodesic $[a, b] \subseteq G$, if $a$ and $b$ are $D$-apart along $\{g^{i} x_{0}\}_{i \in \Z}$, then $[a,b]$ must pass through an $f(D, M)$-neighborhood of $\{g^{i}\}_{i \in \Z}$ \emph{in $G$}.

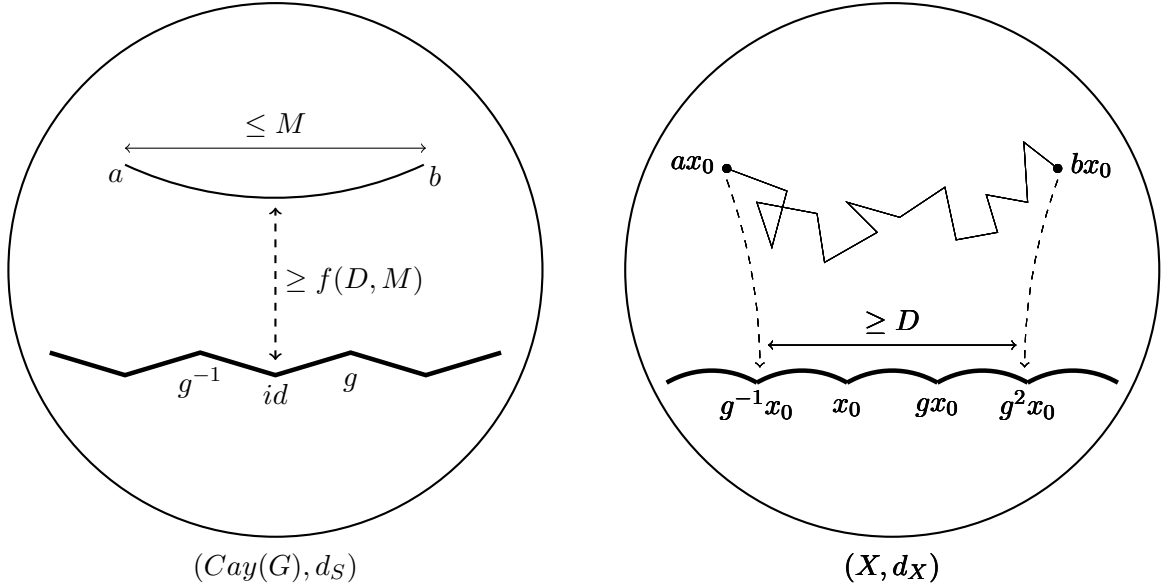
\begin{figure}
\begin{tikzpicture}

\draw[thick] (0, 0.1) circle (3.55);
\draw[line width=0.6mm] (-3, -1)--(-2, -1.3) -- (-1, -1) -- (0, -1.3) -- (1, -1) -- (2, -1.3) -- (3, -1);
\draw[thick] (-2, 1.5) arc (-115:-65:4.7);
\draw[dashed, thick, <->] (0, -1.1) -- (0, 0.93);

\draw (0, -1.55) node {$id$};
\draw (1, -1.4) node {$g$};
\draw (-1, -1.4) node {$g^{-1}$};

\draw (1.05, -0.05) node {$\ge f(D, M)$};
\draw[<->] (-2, 1.72) -- (2, 1.72);
\draw (0, 2.03) node {$\le M$};
\draw (-2.12, 1.35) node {$a$};
\draw (2.12, 1.35) node {$b$};

\draw (0, -3.86) node {$(Cay(G), d_{S})$};

\begin{scope}[shift={(8.2, 0)}]
\draw[thick] (0, 0.1) circle (3.55);

\foreach \i in {-2, ..., 2}{
\draw[line width=0.6mm] (1.2*\i-0.6, -1.4) arc (120:60:1.2);
\draw (-2.2, 1.45) -- (-1.4, 1.15) -- (-1.6, 0.4) -- (-1.8, 1) -- (-1, 0.85) -- (-0.9, 0.2) -- (-0.2, 0.6) -- (-0.6, 1) -- (0.1, 0.8) -- (0.7, 1.2) -- (0.85, 0.5) -- (1.4, 0.6) -- (1.25, 1.1) -- (1.8, 1) -- (1.75, 1.8) --  (2.2, 1.45);

\draw[dashed, ->] (-2.2, 1.3) arc (20:0:7.3);

\draw[dashed, ->] (2.2, 1.3) arc (160:180:7.3);

\fill (-2.2, 1.45) circle (0.06);
\fill (2.2, 1.45) circle (0.06);
\draw[<->] (-1.65, -0.9) -- (1.65, -0.9);

\draw (0.6, -1.73) node {$gx_{0}$};
\draw (-0.6, -1.73) node {$x_{0}$};
\draw (-1.8, -1.7) node {$g^{-1}x_{0}$};
\draw (1.8, -1.7) node {$g^{2}x_{0}$};

\draw (0, -0.6) node {$\ge D$};

\draw (-2.65, 1.5) node {$ax_{0}$};
\draw (2.65, 1.5) node {$bx_{0}$};

\draw (0, -3.86) node {$(X, d_{X})$};
}

\end{scope}

\end{tikzpicture}
\caption{Schematics for $f(D, M)$ in Subsection \ref{subsection:compare}}
\label{fig:schem}
\end{figure}

First, $F_{2}$ has proper action on its own Cayley graph $Cay(F_{2})$. This implies that any coarse stabilizer of $v \in F_{2}$ is finite. Furthermore, each $g \in F_{2} \setminus \{id\}$ has the so-called \emph{strong contracting property}: if a geodesic $[a, b] \subseteq F_{2}$ makes nontrivial progress along $\{g^{i} \}_{i \in \Z}$, then $[a, b]$ passes through a bounded neighborhood of $\{g^{i}\}_{i \in \Z}$. In other words, $f(D, M)$ is constant in $M$ for large enough $D$.

Second, $\Mod(\Sigma)$ acts on the ambient curve complex $\mathcal{C}(\Sigma)$ and tuples of subsurface curve complexes $\mathcal{C}(U), U \subsetneq \Sigma$. Fixing a simple closed curve $x_{0} \in \mathcal{C}(\Sigma)$, each $g \in \Mod(\Sigma)$ gives rise to shadows $d_{U}(x_{0}, gx_{0})$ on various $\{\mathcal{C}(U) : U \subseteq \Sigma\}$, using which the word metric on $\Mod(\Sigma)$ can be (coarsely) estimated via distance formula \cite{masur2000curve}. One consequence of the distance formula and is the \emph{weakly contracting property} of pseudo-Anosov orbits \cite{behrstock2006asymptotic}, \cite{duchin2009divergence}. Explicitly, for each pseudo-Anosov mapping class $g$,  there exists $\epsilon>0$ such that if an $M$-short geodesic $[a, b] \subseteq \Mod(\Sigma)$ makes progress $D$ along $\{g^{i}\}_{i \in \Z}$, then $[a, b]$  passes through a $f(D, M) := (M \cdot e^{- \epsilon D})$-neighborhood of $\{g^{i}\}_{i \in \Z}$. 

There is no direct analogue of the distance formula for $\Out(F_{N})$. As a result, we do not know whether fully irreducible outer automorphisms (which are analogues of pseudo-Anosovs) have the weakly contracting property on the Cayley graph of $\Out(F_{N})$. However, every fully irreducible outer automorphism $g$ has the WPD property (for various hyperbolic actions, cf. \cite{bestvina2010hyp}, \cite{mann2014some}, \cite{bestvina2014hyperbolicity}), i.e., the joint coarse stabilizer of $g^{i}$ and $g^{j}$ is finite when $|i-j|$ is large. This leads to an implicit contracting property, i.e., for every $D$ and $M$ the value of $f(D, M)$ is finite.

Finally, consider a trivial projection of $F_{2} \times F_{3}$ onto the first factor $F_{2}$. This gives rise to a natural action of $F_{2} \times F_{3}$ on $Cay(F_{2})$. This action has not only a large point stabilizer, but also a large global stabilizer. Namely, $\{id\} \times F_{3}$ acts trivially on $Cay(F_{2})$. In addition, there is no contraction along loxodromics on $F_{2}$, i.e., $f(D, M) = +\infty$. In general, if $X \times Y$ is a product space, a $D$-long geodesic $\gamma$ can have $D$-large projection onto $X \times \{id\}$, regardless of the distance of $\gamma$ from $X \times \{id\}$. Figure \ref{fig:tableGp} summarizes the discussion so far.

\begin{figure}\begin{center}
\begin{tabular}{|c ||c c c|} 
 \hline 
 & $\delta$-hyperbolic space & $f(D, M)$ for a fixed $D$ & Density of non-loxodromics \\
 \hline \hline 
 $F_{2}$ & $Cay(F_{2})$ & constant in $M$ & $\lesssim \lambda^{-n}$ for some $\lambda >1$\\
 $\Mod(\Sigma)$ & $\mathcal{C}(\Sigma)$ & linear in $M$ & $\lesssim n^{-k}$ ($\forall k$) \\
 $\Out(F_{N})$ & $\mathcal{FF}_{N}$ & finite& tends to 0 \\
 $F_{2} \times F_{3}$ & $Cay(F_{2})$ & $+\infty$ & can be bounded away from 0 \\
 \hline
\end{tabular}
\end{center}

\caption{Properties of the four actions and the density estimates}
\label{fig:tableGp}
\end{figure}

The more information we have about the growth of $f(D, M)$, the better asymptotics of the density of non-loxodromics we can prove. In $F_{2}$, the proportion of non-loxodromic elements in $B_{S}(n)$ decays exponentially fast in $n$. This is proved by W. Yang \cite{yang2020genericity} in groups with strongly contracting elements, including relatively hyperbolic groups and small cancellation groups.

For the mapping class group, the function $f(D, M)$ grows at most linearly in $M$. Using this property, it is shown in \cite{choi2024counting} that the density of non-pseudo-Anosovs in $B_{S}(n)$ decays faster than $n^{-k}$ for any $k>0$. Similar growth behaviour of $f(D, M)$ is observed in HHGs with Morse elements, because loxodromics on the top curve space have the weakly contracting property. Rank-1 CAT(0) groups also fall into this category, as the strongly contracting property of a rank-1 element on the CAT(0) space implies its weak contracting property in the group.

Without control of $f(D, M)$, loxodromics can either be generic or non-generic depending on the generating set $S$. Indeed, there exist two finite generating sets $S$ and $S'$ of $F_{2} \times F_{3}$, such that loxodromics (for the action on $Cay(F_{2})$) are generic in $S$ but not in $S'$. We refer readers to \cite[Example 1]{gekhtman2018counting}. This simple example also tells us that genericity of Morse elements of a group may not be preserved through a quasi-isometry.

This paper deals with $\Out(F_{N})$ and others of its ilk. There is no \emph{a priori} control on the growth of $f(D, M)$ for acylindrically hyperbolic groups. Our main point is that, nonetheless, the finiteness of $f(D, M)$ is sufficient to conclude the genericity of loxodromics.

\subsection{Another side of the story: random walks} \label{subsection:RW}

There are two popular models to sample a random element in a group $G$. One is the counting method as in Theorem \ref{thm:main}. Namely, we consider a large word metric ball and choose an element with respect to the uniform measure. The other one is the random walk model: we  put a probability measure $\mu$ on a generating set $S$ of $G$ (e.g., the uniform measure when $S$ is finite) and investigate its $n$-fold convolution $\mu^{\ast n}$.

For example, given a $G$-action on a Gromov hyperbolic space $X$, one can ask if  $\Prob_{\mu^{\ast n}} (\textrm{$g$ is loxodromic})$ converges to 1 as $n$ tends to infinity. This is closely related to a description of a typical sample path drawn on $X$, called \emph{ray approximation} or \emph{geodesic tracking}, that was pursued for word hyperbolic groups by V. Kaimanovich \cite{kaimanovich1994poisson}; see \cite{kaimanovich2000hyp} also. It was J. Maher's observation that neither the properness of $X$ nor the properness of the action is necessary. As a result, Maher proved in \cite{maher2011random} that $\Prob_{\mu^{\ast n}}(\textrm{$g$ is pseudo-Anosov})$ converges to 1 in the mapping class group (I. Rivin independently proved this result using different method in \cite{rivin2008walks}).

Maher's observation was later generalized by D. Calegari and J. Maher \cite{calegari2015scl}, and once again by J. Maher and G. Tiozzo in \cite{maher2018random}: they proved that $\Prob_{\mu^{\ast n}} (\textrm{$g$ is loxodromic})$ converges to 1 as long as the $G$-action on $X$ is non-elementary (i.e., $S$ generates two independent loxodromics). In particular, random walks do not care if the group has a large subgroup with trivial action, given that they hit non-elementary loxodromic elements for a positive probability. Maher-Tiozzo's result indeed applies to all 4 group actions in Subsection \ref{subsection:compare}.

Consequently, for the uniform measure $\mu_{S}$ on a finite generating set $S$ of $G$, the genericity of loxodromics with respect to $\mu_{S}^{\ast n}$ does not imply the genericity with respect to $(\textrm{uniform measure on $B_{S}(n)$})$. This is anticipated by the fact that the two measures differ by an exponential factor in $n$.

If one is allowed to pick their favorite generating set $S$ for $G$, then one can bring the estimates from random walks to the counting problem. This was indeed the strategy of \cite{choi2024pseudo-Anosovs}, where the author proved that every finitely generated weakly hyperbolic group has a finite generating set $S$ for which loxodromics are generic. Since the asymptotic density may depend on the choice of $S$ (as shown in \cite[Example 1]{gekhtman2018counting}), this strategy does not establish Theorem \ref{thm:main}.

\subsection{Beyond hyperbolic spaces} \label{subsection:hyperbolic}

The method for Theorem \ref{thm:mainWPD} does not require global hyperbolicity of the space $X$. It only uses the \emph{strongly contracting property} and the WPD property of $g \in G$ in $X$. For simplicity, however, we will not pursue this generality.  It should be noted that the previous assumption does not imply that $g$ is strongly contracting \emph{in $G$}, i.e., with respect to the word metric. For example, the author does not know whether fully irreducibles are weakly contracting with respect to the word metrics (cf. \cite[Question 6.8]{behrstock2014divergence}).

For example, the method for Theorem \ref{thm:mainWPD} applies to finitely generated groups acting on a CAT(0) space (not necessarily cocompactly) that involves a rank-1 isometry with the WPD property. The study of strongly contracting isometries and their dynamics is growing rapidly. We refer the readers to the references in \cite{arzhantseva2015growth},   \cite{yang2019statistically}, \cite{yang2020genericity}, \cite{coulon2022patterson}, \cite{sisto2023morse}, \cite{ding2025growth}.

In fact, the very notion of acylindrically hyperbolic group was already formulated in terms of contracting elements by A. Sisto \cite{sisto2018contracting}, who generalized Maher-Tiozzo's random walk theory  in \cite{maher2018random} to non-hyperbolic spaces. We also note a recent construction by H. Petyt and A. Zalloum \cite[Theorem B]{petyt2024constructing} that justifies why it suffices to consider WPD action on hyperbolic spaces.

\subsection{Open questions}\label{subsection:question}

The methods in \cite{choi2024counting} and this paper still do not answer:

\begin{question}\label{qn:exponential}
Are pseudo-Anosovs exponentially generic in every word metric on $\Mod(\Sigma)$?
\end{question}

There are two types of word metrics for which exponential genericity of pseudo-Anosovs is known. One comes from generating sets mostly consisting of independent pseudo-Anosovs \cite{choi2024pseudo-Anosovs}. The other recent one is due to L. Ding, D. Mart{\'i}nez-Granado and A. Zalloum \cite{ding2025growth}, where the authors consider the $\Mod(\Sigma)$-action on an injective metric space $(Y, d_{Y})$ and collect orbit points in a large $d_{Y}$-ball. It seems hard to push either method  to handle  arbitrary word metric.

For non-HHGs, we can ask: 
\begin{question}\label{qn:exponentialOut}
Are fully irreducibles exponentially generic in $\Out(F_{N})$ with respect to every word metric? Or, is it at least true for some $\alpha>0$ that \[
\frac{\#B_{S}(n) \cap \{\textrm{fully irreducibles}\}}{\#B_{S}(n)} \lesssim n^{-\alpha}? \] 
\end{question}
This question might be answered for a given group $G$ whenever we know the growth of the function $f(D, M)$ in Subsection \ref{subsection:compare}. 

One can ask more details about generic elements. In the random walk side we have strong law of large numbers (SLLN): for any non-elementary random walk $(Z_{n})_{n>0}$ on a Gromov hyperbolic space, there exists $\lambda \in (0, +\infty]$ such that $\lim_{n} \frac{\|Z_{n}\|_{X}}{n} = \lim_{n} \frac{\tau_{X}(Z_{n})}{n} = \lambda$ almost surely (\cite{calegari2015scl}, \cite{maher2018random}, \cite{baik2021linear}). Here the key point is the linear growth of displacement and translation length. We pose:

\begin{question}\label{qn:linear}
Do generic fully irreducibles have linearly growing translation length? Namely, given a finite generating set $S$ of $\Out(F_{N})$, does there exist a linear function $f : \mathbb{R} \rightarrow \mathbb{R}$ such that  \begin{equation}\label{eqn:linear}
\lim_{n \rightarrow +\infty} \frac{\#B_{S}(n) \cap \big\{g \in \Out(F_{N}): \tau_{\mathcal{FF}}(g) \ge f(n) \big\}}{\#B_{S}(n)} = 1? \end{equation}
\end{question}

Our method does provide a diverging function $f$ for which Equation \ref{eqn:linear} holds, but we have no control on the growth of $f$. For the mapping class group, the author anticipates that the method in \cite{choi2024counting}  guarantees $f(n) \gtrsim \sqrt{n}$. The results of \cite{choi2024pseudo-Anosovs} and \cite{ding2025growth} imply that $f(n) \gtrsim n$ works for \emph{certain} finite generating set $S$.

Finally, we state a question related to Question \ref{qn:exponential}.
\begin{question}\label{qn:exponentialGrowth}
Does $G=\Mod(\Sigma)$ or $G=\Out(F_{N})$ have purely exponential growth? That means, for (some or every) finite generating set $S$ of $G$, does there exist $K, \lambda > 1$ such that \[
\frac{1}{K} \lambda^{n} \le \#B_{S}(n) \le K \lambda^{n}? \quad (\forall n>0)
\]
\end{question}
This question is answered by W. Yang for groups with strongly contracting elements \cite[Theorem B]{yang2019statistically}. Meanwhile, we do not know the answer for $\Mod(\Sigma)$ for any finite generating set.

\subsection{Plans}\label{subsection:plans}

After reviewing preliminaries in Section \ref{section:prelim}, we observe a variant of A. Sisto's geometric separation lemma \cite[Lemma 3.3]{sisto2016quasi-convexity} in Section \ref{section:sublinear}. We then prove the main theorem in Section \ref{section:proof}.

\subsection*{Acknowledgments}
The author thanks Jason Behrstock for discussion about weakly contracting property in the mapping class group, Jason Manning for discussion about acylindrical actions, and Alessandro Sisto for discussion about Lemma \ref{lem:amplifyWPD}.

This work was first conceived during the program ``Randomness and Geometry" at the Fields Institute. The author thanks the Fields Institute and the organizers of the program for their hospitality. Some progress was made while the author was visiting June E Huh Center for Mathematical Challenges at KIAS in January 2025.

\section{Preliminaries}\label{section:prelim}

In this section, we collect some notions and facts about acylindrically hyperbolic groups. We refer to Gromov's seminal paper \cite{gromov1987hyperbolic} and standard textbooks \cite{coornaert1990geometrie}, \cite{ghys1990bord}.

A metric space is said to be \emph{geodesic} if every pair of points can be connected by a geodesic. For two points $x$ and $y$ in this space, we denote by $[x, y]$ an arbitrary geodesic connecting $x$ to $y$. Given $\delta>0$, we say that a geodesic metric space is $\delta$-hyperbolic if every geodesic is $\delta$-slim. 

Given a geodesic $\gamma: I \rightarrow X$, we will sometimes denote the image $Im(\gamma) \subseteq X$ by $\gamma$. Based on this convention, we define the \emph{closest point projection} $\pi_{\gamma} : X \rightarrow 2^{\gamma}$ by \[
y \in \pi_{\gamma}(x) \,\, \Leftrightarrow\,\, d_{X}(x, y) = \inf\big\{ d_{X}(x, p) : p \in \gamma\big\}.
\]
We say that two geodesics  $\gamma : [0, L] \rightarrow X$ and $\eta : [0, L'] \rightarrow X$ are \emph{$\epsilon$-fellow traveling} if \[
 d_{X} (\gamma(0), \eta(0)) < \epsilon,\,\, d_{X}(\gamma(L), \eta(L')) < \epsilon\,\,\textrm{and}\,\,  d_{Haus}(\gamma, \eta) < \epsilon.
\]
The fellow traveling property is transitive: if $\gamma_{1}$ and $\gamma_{2}$ are $\epsilon$-fellow traveling; $\gamma_{2}$ and $\gamma_{3}$ are $\epsilon'$-fellow traveling, then $\gamma_{1}$ and $\gamma_{3}$ are $(\epsilon+\epsilon')$-fellow traveling. Furthermore, we have:
\begin{fact}\label{fact:fellowEndpt}
Let $X$ be a $\delta$-hyperbolic space and let $x, y, z, w \in X$ be such that $d_{X}(x, y)<\epsilon$ and $d_{X}(z, w)< \epsilon'$. Then $[x, z]$ and $[y, z]$ are $(\epsilon + \delta)$-fellow traveling. Moreover, $[x, z]$ and $[y, w]$ are $(\epsilon+\epsilon' + 2\delta)$-fellow traveling.
\end{fact}

For each $x \in X$, $\pi_{\gamma}(x)$ may not be a singleton. Nevertheless, its diameter is bounded and $\pi_{\gamma}(\cdot)$ is coarsely Lipschitz. The following is a consequence of \cite[Proposition 10.2.1]{coornaert1990geometrie}, which follows from the tree approximation lemma \cite[Th{\'e}or{\`e}me 8.1]{coornaert1990geometrie}, \cite[Th{\'e}or{\'e}me 2.12]{ghys1990bord}.

\begin{fact}\label{fact:projectionCoarse}
Let $X$ be a $\delta$-hyperbolic space.
\begin{enumerate}
\item Let $x, y \in X$ and let $\gamma$ be a geodesic in $X$. Then $\pi_{\gamma}(x) \cup \pi_{\gamma}(y)$ has diameter at most $d_{X}(x, y) + 12\delta$.
\item Let $x, y \in X$, let $\gamma$ be a geodesic in $X$ and let $p \in \pi_{\gamma}(x)$ and $q \in \pi_{\gamma}(y)$. Suppose that $p$ appears earlier than $q$ on $\gamma$ and that $d_{X}(p, q) > 20\delta$. Then any geodesic $[x, y]$ between $x$ and $y$ contains a subsegment that is $20\delta$-fellow traveling with $[p, q]$.
\end{enumerate}
\end{fact}

\begin{cor}[{\cite[Lemma 4.1]{sisto2018contracting}}]\label{cor:monotone}
Let $X$ be a $\delta$-hyperbolic space, let $\gamma$ be a geodesic in $X$, let $x, y \in X$ and let $\eta$ be a subsegment of $\gamma$ that contains $\pi_{\gamma}(x) \cup \pi_{\gamma}(y)$. Then 
$\pi_{\gamma}([x, y])$ is contained in the  $60\delta$-neighborhood of $\eta$.
\end{cor}

\begin{proof}
Suppose to the contrary that there exist $z \in [x, y]$, $p \in \pi_{\gamma}(x)$, $q \in \pi_{\gamma}(y)$, $r \in \pi_{\gamma}(z)$ such that $d_{X}(p, r), d_{X}(q, r) \ge 60\delta$ and such that $p, q$ are to the right of $r$. Let $p_{0}$ be the point on $\gamma$ to the right of $r$ such that $d_{X}(r, p_{0}) = 60\delta$. Then Fact \ref{fact:projectionCoarse}(2) implies that there exist a subsegment $[r', p']$ of $[z, x]$ and a subsegment $[r'', p'']$ of $[z, y]$ such that $d_{X}(r', r), d_{X}(r'', r) < 20\delta$ and $d_{X}(p', p_{0}), d_{X}(p'', p_{0}) < 20\delta$. We then observe that \[
\begin{aligned}
40\delta &> d_{X}(p', p_{0}) + d_{X}(p_{0}, p'') \ge d_{X}(p', p'') \ge d_{X}(p', r') + d_{X}(r'', p'') \\
&\ge [d_{X}(p_{0}, r) - d_{X}(p_{0}, p') - d_{X}(r, r')] +  [d_{X}(p_{0}, r) - d_{X}(p_{0}, p'') - d_{X}(r, r'')] > 20\delta + 20\delta,
\end{aligned}
\]
a contradiction. Similar contradiction happens when $p, q$ are both to the left of $r$.
\end{proof}

For $x, y, z \in X$, we define the \emph{Gromov product} of $y$ and $z$ based at $x$ by \[
(y, z)_{x} := \frac{1}{2} \big[ d_{X}(y, x) + d_{X}(x, z) - d_{X}(y, z) \big].
\]
Gromov hyperbolicity has the following consequence.

\begin{fact}[{\cite[Lemma A.3]{choi2024counting}}]\label{fact:thinTri}
Let $X$ be a $\delta$-hyperbolic space. Let $x, y, z \in X$ and let $p \in [y, z]$ be such that $d_{X}(p, y) = (x, z)_{y}$. Then $\pi_{[y, z]}(x)$ is contained in the $8\delta$-neighborhood of $p$.
\end{fact}

\begin{dfn}\label{dfn:WPD}
Let $G$ be a finitely generated group acting on a $\delta$-hyperbolic space $(X, d_{X})$ with a basepoint $x_{0} \in X$. We say that a loxodromic element $\varphi \in G$ has the \emph{WPD (weak proper discontinuity) property} if for each $K$ there exists $N, M$ such that \[
\# \Big(Stab_{K}(x_{0}, \varphi^{N} x_{0}) := \big\{ g \in G : d_{X} (x_{0}, gx_{0}) < K\,\, \textrm{and}\,\,d_{X}(\varphi^{N} x_{0}, g\varphi^{N} x_{0}) < K \big\}\Big) < M.
\]
\end{dfn}

We say that a finitely generated group $G$ is \emph{acylindrically hyperbolic} if it admits an isometric action on a $\delta$-hyperbolic space with a WPD loxodromic element $\varphi \in G$. An acylindrically hyperbolic group $G$ is said to be \emph{non-elementary} if it is not virtually cyclic. The following fact is a consequence of \cite[Proposition 6(1), (2)]{bestvina2002bounded}. The proof is sketched in \cite[Fact 2.2]{choi2024counting}.

\begin{fact}\label{fact:WPDNonElt}
Let $G$ be a non-virtually cyclic group with a generating set $S$. Suppose that $G$ acts on a $\delta$-hyperbolic space $X\ni x_{0}$ with a WPD loxodromic element $\varphi \in G$. Then there exists $E_{0} > 0$ such that the following hold.
\begin{enumerate}
\item For each $g \in G$, there exist $s, t \in S \cup \{id\}$ such that $(\varphi^{i} x_{0}, sgx_{0})_{x_{0}} \le E_{0}$ for all $i>0$ and $(\varphi^{j} x_{0}, tgx_{0})_{x_{0}} \le E_{0}$ for all $j<0$.
\item Let $n>0$ and $g \in G$. Let $\gamma := [x_{0}, \varphi^{n} x_{0}]$, let $p \in \pi_{\gamma} (gx_{0})$ and let $q \in \pi_{\gamma}(g\varphi^{n} x_{0})$. Suppose that $p$ appears earlier than $q$ along $\gamma$ and suppose that $d_{X}(p, q)> E_{0}$. Then $d_{S}(\varphi^{i}, g \varphi^{j}) < E_{0}$ for some $i , j\in \{0, 1,\ldots, n\}$.
\end{enumerate}
\end{fact}

We now recall the notion of alignment.

\begin{dfn}\label{dfn:align}
Let $K>0$ and let $\gamma_{1}, \gamma_{2}, \ldots, \gamma_{n}$ be finite geodesics (which can be degenerate, i.e., points). We say that $(\gamma_{1}, \ldots, \gamma_{n})$ is $K$-aligned if for each $i=1, \ldots, n-1$ we have \[
\begin{aligned}
\diam\big(\pi_{\gamma_{i}}(\gamma_{i+1}) \cup (\textrm{ending point of $\gamma_{i}$})\big) < K \,\,\textrm{and}\\ \diam\big(\pi_{\gamma_{i+1}}(\gamma_{i}) \cup (\textrm{beginning point of $\gamma_{i+1}$})\big) < K.
\end{aligned}
\]
\end{dfn}

The following facts are straightforward, whose proofs can be found in \cite[Appendix]{choi2024counting}.

\begin{fact}\label{fact:fellowAlign}
Let $\gamma$ be a geodesic in a metric space. Let $\gamma_{1}$ and $\gamma_{2}$ be subsegments of $\gamma$, with $\gamma_{1}$ appearing earlier than $\gamma_{2}$. Let $\kappa_{1}$ and $\kappa_{2}$ be geodesics that are $K$-fellow traveling with $\gamma_{1}$ and $\gamma_{2}$, respectively. Then $(\kappa_{1}, \kappa_{2})$ is $6K$-aligned.
\end{fact}

\begin{fact}\label{fact:fellowOverlap}
The following holds for each $K>0$ and $L \ge 12K$. Let $\gamma$ be a geodesic in a metric space and let $\gamma_{1}$ and $\gamma_{2}$ be subsegments of $\gamma$ such that $\gamma_{1} \cap \gamma_{2}$ has length $L$. Let $[x, y]$ and $\kappa_{2}$ be geodesics that are $K$-fellow traveling with $\gamma_{1}$ and $\gamma_{2}$, respectively. Then $\pi_{\kappa}(x)$ appears earlier than $\pi_{\kappa}(y)$ along $\kappa$, and $d_{X} (\pi_{\kappa}(x), \pi_{\kappa}(y)) > L - 10K$.
\end{fact}

We now record a version of Behrstock's inequality \cite[Theorem 4.3]{behrstock2006asymptotic} (cf. \cite[Lemma 2.5]{sisto2018contracting}) and its consequences. The proofs can be found in \cite[Section 3, Appendix]{choi2024counting}.

\begin{fact}\label{fact:Behr}
Let $X$ be a $\delta$-hyperbolic space. Let $x \in X$ and let $(\gamma_{1}, \gamma_{2})$ be a $K$-aligned sequence of geodesics in $X$. Then either $(x, \gamma_{2})$ is $(K+60\delta)$-aligned or $(\gamma_{1}, x)$ is $(K+60\delta)$-aligned.
\end{fact}

\begin{fact}\label{fact:Behrstock}
Let $X$ be a $\delta$-hyperbolic space. Let $n \ge 3$ and let $(\gamma_{1}, \ldots, \gamma_{n})$ be a $K$-aligned sequence of geodesics in $X$. Suppose that $\gamma_{2}, \ldots, \gamma_{n-1}$ are longer than $2K + 120\delta$. Then $(\gamma_{i}, \gamma_{j})$ is $(K+60\delta)$-aligned for each $1 \le i < j \le n$.
\end{fact}

\begin{fact}\label{fact:GromProdFellow}
Let $X$ be a $\delta$-hyperbolic space. Let $x, y \in X$ and let $\gamma_{1}, \ldots, \gamma_{n}$ be geodesics in $X$, longer than $2K+140\delta$ each, such that $(x, \gamma_{1}, \ldots, \gamma_{n}, y)$ is $K$-aligned.

Then there exist disjoint subsegments $\eta_{1}, \ldots, \eta_{n}$ of $[x, y]$ such that \begin{enumerate}
\item $\eta_{1}, \ldots, \eta_{n}$ are in order from left to right along $[x, y]$, i.e., $\eta_{i}$ appears earlier than $\eta_{i+1}$ along $[x, y]$ for each $i=1, \ldots, n-1$, and
\item $\gamma_{i}$ and $\eta_{i}$ are $(K+80\delta)$-fellow traveling for each $i=1, \ldots, n$.
\end{enumerate}
\end{fact}

Let $G$ be a group and let $S \subseteq G$ be its finite generating set. The word metric $d_{S}$  is defined by \[
d_{S}(g, h) := \min \left\{ n \in \Z_{\ge 0} : \begin{array}{c} \textrm{$\exists \,a_{1}, a_{2}, \ldots, a_{n} \in S, \,\epsilon_{1},\epsilon_{2}, \ldots, \epsilon_{n} \in \{1, -1\}$} \\  \textrm{such that $g^{-1} h = a_{1}^{\epsilon_{1}} a_{2}^{\epsilon_{2}} \cdots a_{n}^{\epsilon_{n}}$}.\end{array}\right\}
\]
We use the notation for the word norm $\|g\|_{S} := d_{S}(id, g)$. We define \[
\begin{aligned}
B_{S}(n) &:= \big\{g \in \Mod(\Sigma) : d_{S}(id, g) \le n\big\}.
\end{aligned}
\]
We denote by $[g, h]_{S}$ an arbitrary $d_{S}$-geodesic between $g, h \in G$. By a $d_{S}$-path, we mean a sequence of group elements $P = (g_{1}, g_{2}, \ldots, g_{n})$ such that $d_{S}(g_{i} ,g_{i+1}) = 1$ for each $i$; we denote $n$ by $Len(P)$.

When the group $G$ acts on a metric space $X \ni x_{0}$, we often define \[
\begin{aligned}
\|g\|_{X} &:= d_{X}(x_{0}, gx_{0}) \quad (g \in G), \\
K_{Lip} &:= \max_{s \in S} \|s\|_{X}.
\end{aligned}
\]
Then we have $\|g\|_{X} \le K_{Lip} \|g\|_{S}$ for each $g \in G$.

\section{WPD property and contraction}\label{section:sublinear}

It is well-known that a loxodromic isometry $\varphi$ of a $\delta$-hyperbolic space $X \ni x_{0}$ has strictly positive asymptotic translation length $\tau:= \lim_{n} d_{X}(x_{0}, \varphi^{n} x_{0})/n$. Moreover, its orbit $\{\varphi^{i} x_{0}\}_{i\in\Z}$ is a quasigeodesic and hence quasi-convex. In summary,
\begin{fact}\label{fact:morseLox}
Let $\varphi$ be a loxodromic isometry of a $\delta$-hyperbolic space $X \ni x_{0}$. Then there exists $\mathcal{G}>0$ such that the sequence $(\varphi^{i} x_{0}, \ldots, \varphi^{j} x_{0})$ and the geodesic $[\varphi^{i} x_{0},  \varphi^{j} x_{0}]$ are $\mathcal{G}$-fellow traveling for each $i \le j$. Furthermore, the sequence $(\varphi^{i} x_{0})_{i \in \Z}$ is a $\mathcal{G}$-coarse geodesic, i.e., \[
d_{X}(\varphi^{i} x_{0}, \varphi^{l} x_{0}) \ge d_{X}(\varphi^{i} x_{0}, \varphi^{j} x_{0}) +  d_{X}(\varphi^{j} x_{0}, \varphi^{l} x_{0}) - \mathcal{G} \quad (\forall i \le j \le l).
\]
\end{fact}

In Subsection \ref{subsection:compare}, we claimed that $f(D, M) <+\infty$ for each $D, M > 0$ for every acylindrically hyperbolic group. We prove a variant of this fact.

\begin{lem}\label{lem:amplifyWPD}
Let $G$ be a non-virtually cyclic group with a finite generating set  $S \subseteq G$. Suppose that $G$  acts on a $\delta$-hyperbolic space $X \ni x_{0}$ with a WPD loxodromic element $\varphi \in G$. Then there exists $D_{0} > 0$, and for each $k, M>0$ there exists $R=R(k, M)>0$, such that the following holds.

Let $g, h \in G$ be such that $\|g\|_{S} > R$ and $\|h\|_{S} \le M$. Then $\pi_{[x_{0}, \varphi^{k} x_{0}]} (\{gx_{0}, ghx_{0}\})$ has diameter at most $D_{0}$. 
\end{lem}

This lemma closely resembles \cite[Lemma 3.3]{sisto2016quasi-convexity} and \cite[Lemma 8.1]{mathieu2020deviation}. Here, the crucial point is that $D_{0}$ is uniform and is independent from $k, M$ and $R$.

\begin{proof}
Let $\mathcal{G}>0$ be the constant for $\varphi$ as in Fact \ref{fact:morseLox}. For $K = 24\mathcal{G} + 130\delta$, we pick $N$ such that $Stab_{K}(x_{0}, \varphi^{N} x_{0})$ is finite using the WPD property of $\varphi$. We then set $D_{0} := 1002\mathcal{G} + N D_{\varphi} + 1000\delta$, where $D_{\varphi} := d_{X}(x_{0}, \varphi x_{0})$.

To prove the lemma, let $k, M>0$ and denote $\gamma := [x_{0}, \varphi^{k} x_{0}]$. Suppose to the contrary that there does not exist $R$ for $(k, M)$. That means, suppose that there exist a sequence $(g_{1}, g_{2}, \ldots )$ of distinct elements of $G$ and a sequence $(h_{1}, h_{2}, \ldots)$ in $B_{S}(M)$ such that \[
\diam\big( \pi_{\gamma} (\{g_{i}x_{0}, g_{i}h_{i}x_{0}\})\big) \ge D_{0} \quad (\forall i >0).
\]
Let $p_{i}, q_{i}$ be points in $\pi_{\gamma} (\{g_{i}x_{0}, g_{i}h_{i}x_{0}\})$ that are at least $D_{0}$-apart. Recall that the nearest point projection of a single point onto $\gamma$ has diameter at most $20\delta < D_{0}$ (Fact \ref{fact:projectionCoarse}(1)). Hence, up to relabelling, we can say that  $p_{i} \in \pi_{\gamma}(g_{i}x_{0})$ and $q_{i} \in \pi_{\gamma}(g_{i}h_{i}x_{0})$.

Since $\gamma = [x_{0}, \varphi^{k} x_{0}]$ is compact and $B_{S}(M)$ is finite, we can take a subsequence and assert that:  \[\begin{aligned}
h_{1} = h_{2} = \ldots &=: h, \\ d_{X}(p_{i}, p_{j}), d_{X}(q_{i}, q_{j}) &< \mathcal{G} \quad (\forall i, j >0).
 \end{aligned}
\]
Either $p_{i}$'s appear earlier than or later than $q_{i}$'s. In the latter case, we can replace $h$ with $h^{-1}$, $g_{i}$ with $g_{i} h$ and swap $p_{i}$'s with $q_{i}$'s. Hence, we may assume that $p_{i}$'s appear earlier than $q_{i}$'s.

By Fact \ref{fact:projectionCoarse}(2), there exists $[\alpha_{i}, \beta_{i}] \subseteq [g_{i} x_{0}, g_{i} hx_{0}]$ that is $20\delta$-fellow traveling with $[p_{i}, q_{i}]$. Note that $0 \le d_{X}(g_{i} x_{0}, \alpha_{i}) \le d_{X}(x_{0}, hx_{0})$. By taking further subsequence, we can obtain $T$ such that \[
\big|d_{X}(g_{i} x_{0}, \alpha_{i}) - T\big| < \delta \quad (\forall i > 0).
\]

By Fact \ref{fact:morseLox}, there exist $l, m \in \Z$ such that $d_{X}(\varphi^{l} x_{0}, p_{1}) < \mathcal{G}$ and $d_{X}(\varphi^{m} x_{0}, q_{1}) < \mathcal{G}$. Note that \begin{equation}\label{eqn:DVarphi}
D_{\varphi} \cdot |l-m| \ge d_{X}(\varphi^{l}x_{0}, \varphi^{m} x_{0}) > d_{X}(p_{1}, q_{1}) - 2\mathcal{G} > 1000\mathcal{G} + N D_{\varphi}.
\end{equation}
Here, if $m \le l$ then \[\begin{aligned}
d_{X}(x_{0}, q_{1}) &\le d_{X}(x_{0}, \varphi^{m} x_{0}) + \mathcal{G} \le d_{X}(x_{0}, \varphi^{l} x_{0}) - d_{X}(\varphi^{m} x_{0}, \varphi^{l} x_{0}) + 2\mathcal{G}\\
&\le d_{X}(x_{0}, \varphi^{l} x_{0}) - 1000\mathcal{G} + 2\mathcal{G} < d_{X}(x_{0}, \varphi^{l}x_{0}) - 998\mathcal{G} \le d_{X}(x_{0}, p_{1}) - 997\mathcal{G}.
\end{aligned}
\]
This contradicts the fact that $p_{1}$ appears earlier than $q_{1}$ on $\gamma$. Hence, we have $l < m$. 

Now Inequality \ref{eqn:DVarphi} implies that $l + N$ lies between $l$ and $m$. By Fact \ref{fact:morseLox}, $\varphi^{l+N} x_{0}$ lies in a $\mathcal{G}$-neighborhood of $[\varphi^{l} x_{0}, \varphi^{m}x_{0}]$. Note that  $d_{X}(\varphi^{l} x_{0}, p_{i}) \le d_{X}(\varphi^{l}x_{0}, p_{1}) + d_{X}(p_{1}, p_{i}) \le 2\mathcal{G}$ for each $i$, and similarly $\varphi^{m}x_{0}$ and $q_{i}$ are $2\mathcal{G}$-close.

By Fact \ref{fact:fellowEndpt} $[\varphi^{l} x_{0}, \varphi^{m}x_{0}]$ is $(4\mathcal{G}+2\delta)$-fellow traveling with $[p_{i}, q_{i}]$, which is $20\delta$-fellow traveling with $[\alpha_{i}, \beta_{i}]$. Thus, there exists $c_{i} \in [\alpha_{i}, \beta_{i}]$ such that $d_{X}(c_{i}, \varphi^{l+N} x_{0}) < 5\mathcal{G} + 22\delta$. We have \[\begin{aligned}
\big|d_{X}(\alpha_{i}, c_{i}) - d_{X}(\varphi^{l} x_{0}, \varphi^{l+N} x_{0}) \big| &\le d_{X}(\alpha_{i}, \varphi^{i} x_{0}) + d_{X}(c_{i}, \varphi^{l+N} x_{0})\\
&\le d_{X}(\alpha_{i}, p_{i}) + d_{X}(p_{i}, p_{1}) + d_{X}(p_{1}, \varphi^{l} x_{0}) + d_{X}(c_{i}, \varphi^{l+N} x_{0}) \\
&\le 20\delta + \mathcal{G} + \mathcal{G} + (5\mathcal{G}+22\delta) \le 7\mathcal{G}+42\delta.
\end{aligned}
\]

Now, for each $i$ we have four points $g_{i} g_{1}^{-1} \alpha_{1}$, $g_{i} g_{1}^{-1} c_{1}$, $\alpha_{i}$  on the geodesic \[
g_{i} \cdot g_{1}^{-1}([g_{1} x_{0}, g_{1} hx_{0}]) = [g_{i} x_{0}, g_{i} h x_{0}].
\]
 Recall that $d_{X}(g_{i} x_{0}, g_{i} g_{1}^{-1} \alpha_{1}) = d_{X}(g_{1} x_{0}, \alpha_{1})$ and $d_{X}(g_{i} x_{0}, \alpha_{i})$ are both $\delta$-close to $T$. This implies that $g_{i} g_{1}^{-1} \alpha_{1}$ and $\alpha_{i}$ are $2\delta$-close. Hence, we have \[
\begin{aligned}
d_{X}(\varphi^{l} x_{0}, g_{i} g_{1}^{-1} \cdot \varphi^{l} x_{0}) &\le 
d_{X}(\varphi^{l} x_{0},\alpha_{i}) + d_{X}(\alpha_{i}, g_{i} g_{1}^{-1} \alpha_{1}) + d_{X}(g_{i} g_{1}^{-1} \alpha_{1},  g_{i} g_{1}^{-1} \cdot \varphi^{l} x_{0}) \\
&\le d_{X}(\varphi^{l}x_{0}, p_{i}) + d_{X}(p_{i}, \alpha_{i}) + 2\delta + d_{X}(\varphi^{l} x_{0}, p_{1}) + d_{X} (p_{1}, \alpha_{1}) \\
&\le (20\delta + 2\mathcal{G}) + 2\delta + (\mathcal{G}+20\delta) = 42\delta + 3\mathcal{G}.
\end{aligned}
\]
Next, $d_{X}(g_{i} x_{0}, c_{i}) = d_{X}(g_{i} x_{0}, \alpha_{i}) + d_{X}(\alpha_{i} ,c_{i})$ is $(7\mathcal{G}+43\delta)$-close to $T + d_{X}(x_{0}, \varphi^{N} x_{0})$. So is $d_{X}(g_{i} x_{0}, g_{i} g_{1}^{-1} c_{1}) = d_{X}( g_{1}x_{0}, c_{1})$. Hence, $c_{i}$ and $g_{i} g_{1}^{-1 }c_{1}$ are $(14\mathcal{G} + 86\delta)$-close. We conclude \[\begin{aligned}
d_{X}(\varphi^{l+N} x_{0}, g_{i} g_{1}^{-1} \cdot \varphi^{l+N} x_{0})& \le 
d_{X}(\varphi^{l+N} x_{0},c_{i}) + d_{X}(c_{i}, g_{i} g_{1}^{-1} c_{1}) + d_{X}(g_{i} g_{1}^{-1} c_{1},  g_{i} g_{1}^{-1} \cdot \varphi^{l+N} x_{0}) \\
&\le (5\mathcal{G}+22\delta) + (14\mathcal{G}+86\delta) + (5\mathcal{G} + 22\delta) < 24\mathcal{G} + 130 \delta.
\end{aligned}
\]
To summarize, $\varphi^{-l} g_{i} g_{1}^{-1} \varphi^{l}$ belongs to $Stab_{K}(x_{0}, \varphi^{N} x_{0})$ for each $i$. Furthermore, we have \[
\varphi^{-l} g_{i} g_{1}^{-1} \varphi^{l} = \varphi^{-l} g_{j} g_{1}^{-1} \varphi^{l} \,\,\Leftrightarrow \,\, g_{i} = g_{j}.
\]
Since $g_{1}, g_{2}, \ldots$ are distinct, it follows that $Stab_{K}(x_{0}, \varphi^{N} x_{0})$ is infinite, a contradiction.
\end{proof}

\begin{prop}\label{prop:subContN}
Let $G$ be a non-virtually cyclic group and let $S \subseteq G$ be its finite generating set. Suppose that $G$ acts on a $\delta$-hyperbolic space $X \ni x_{0}$ with a WPD loxodromic element $\varphi \in G$. Then for each $K>0$ there exists $L_{0} = L_{0}(K)$ such that, for each $L\ge L_{0}$ and for each $M>0$ there exists $R_{0} =R_{0}(L, M)> 0$ satisfying the following.

Let $P_{l}$ be a $d_{S}$-path connecting $a_{l} \in G$ to $b_{l}\in G$ for $l=1, 2$. Let $g_{1}, \ldots, g_{m}\in G$ be such that \[
(a_{i}x_{0}, g_{1}[x_{0}, \varphi^{L} x_{0}], \ldots, g_{m}[x_{0}, \varphi^{L} x_{0}], b_{i}x_{0})\textrm{ is $K$-aligned.} \quad (i=1, 2)
\]
Let $k \le m$. Then one of the following happens: \begin{enumerate}
\item For each subset $I \subseteq \{1, \ldots, m\}$ of cardinality $k$, there exist $i \in I$ such that \[
d_{S}(P_{l}, g_{i}) < R_{0} \quad( l=1, 2)\textrm{; or,}
\]
\item $Len(P_{1}) + Len(P_{2}) \ge M \cdot k$.
\end{enumerate}

\begin{proof}
Let $D_{0}$ be as in Lemma \ref{lem:amplifyWPD}. Let $K_{Lip} = \max_{s \in S} \|s\|_{X}$, let $\tau := \lim_{n} \frac{1}{n} d_{X}(x_{0}, \varphi^{n} x_{0})$ and let  \[
L_{0} = \frac{1}{\tau} \big( 2K+2K_{Lip}  + 1000\delta +  D_{0}\big).
\]
Now given $L > L_{0}$, we will declare $R_{0} = R_{0} (L, M)$ following Lemma \ref{lem:amplifyWPD}.

Consider $P_{1}, P_{2}$ and $g_{i}$'s as in the proposition. Note that \begin{equation}\label{eqn:longLnth}
\diam_{X}([x_{0}, \varphi^{L}x_{0}]) \ge \tau L \ge 2K +2K_{Lip} + 1000\delta + D_{0}.
\end{equation}
Thanks to this inequality, we can apply Fact \ref{fact:Behrstock} to the $K$-aligned translates of $[x_{0}, \varphi^{L}x_{0}]$.

We will negate the case (1) and prove that (2) holds. For this, let $I \subseteq \{1, \ldots, m\}$ be a $k$-element subset such that $g_{i}$ is not simultaneously $R_{0}$-close to $P_{1}$ and $P_{2}$ for each $i \in I$. Let \[
I_{l} := \{ i : d_{S}(P_{l}, g_{i}) \ge R_{0} \}. \quad (l=1, 2)
\]
Then $I_{1} \cup I_{2}$ has cardinality at least $k$. For convenience, we will denote $g_{i}[x_{0}, \varphi^{L}x_{0}]$ by $\gamma_{i}$.

For each $i \in \{1, \ldots, m\}$, let $p_{i}$ be the latest vertex of $P_{1}$ such that $(p_{i} x_{0}, \gamma_{i})$ is $(K+60\delta)$-aligned. Such a vertex exists because $(a_{i} x_{0},  \gamma_{i})$ is $(K + 60\delta)$-aligned by Fact \ref{fact:Behrstock}. Now let\[
V_{i} := \{v \in P_{1} : v \,\,\textrm{comes later than $p_{i}$ along $P_{1}$ and $(\gamma_{i}, vx_{0})$ is $(K + 60\delta)$-aligned}\};
\]
$V_{i}$ is nonempty because $(\gamma_{i}, b_{i})$ is $(K+60\delta)$-aligned by Fact \ref{fact:Behrstock}. Let $q_{i}$ be the earliest vertex in $V_{i}$. Lastly, let $p_{i}'$ be the vertex of $P_{1}$ right after $p_{i}$, and let $q_{i}'$ be the one right before $q_{i}$.

Note that $\pi_{\gamma_{i}} (\{p_{i} x_{0}, p_{i}' x_{0}\})$ has diameter at most $K_{Lip} + 12\delta$ by Fact \ref{fact:projectionCoarse}(1). Hence, $\pi_{\gamma_{i}}(p_{i}'x_{0})$ is contained in the beginning $(K_{Lip} + K + 80\delta)$-subsegment of $\gamma_{i}$ and does not meet the ending $(K_{Lip} + 60\delta)$-subsegment (Inequality \ref{eqn:longLnth}). This implies $p_{i}' \notin V_{i}$, and $q_{i}$ comes later than $p_{i}'$. Let \[
V_{i}' := \{v \in P_{1} : v \,\, \textrm{in between $p_{i}$ and $q_{i}$ (excluding $p_{i}, q_{i}$)}\}.
\]
We have then observed that $p_{i}', q_{i}' \in V_{i}'$ and $V_{i}'$ is nonempty. For each $v \in V_{i}'$ we conclude that \[
\textrm{neither $(\gamma_{i}, vx_{0})$ nor $(vx_{0}, \gamma_{i})$ is $(K + 60\delta)$-aligned.}
\]
Now repeated application of Fact \ref{fact:Behr} tells us that for $i \neq j$, \[
v \in V_{i}' \Rightarrow \left\{ \begin{array}{cc} \textrm{$(\gamma_{j}, vx_{0})$ is $(K+60\delta)$-aligned} &\textrm{if}\,\, j < i \\ \textrm{$(vx_{0}, \gamma_{j})$ is $(K+60\delta)$-aligned} &\textrm{if}\,\, j > i \end{array}\right\} \Rightarrow v \notin V_{j}'.
\]
In conclusion, $V_{1}', \ldots, V_{m}'$'s are disjoint subpaths of $P_{1}$.

Now observe for each $i$ that \[\begin{aligned}
\diam_{X} (\pi_{\gamma_{i}} (\{p_{i}', q_{i}'\})) &\ge \diam(\gamma_{i}) - 2(K + 60\delta) - \diam(\pi_{\gamma_{i}}(p_{i}x_{0}, p_{i}'x_{0})) - \diam(\pi_{\gamma_{i}}(q_{i}x_{0}, q_{i}'x_{0})) \\&\ge \tau L - 2(K + 60\delta) - 2 \cdot (K_{Lip} + 20\delta) > D_{0}.
\end{aligned}
\]
If $i \in I_{1}$, we additionally know that $d_{S}(p_{i}', g_{i}) \ge R_{0}$. Lemma \ref{lem:amplifyWPD} then implies that $d_{S} (p_{i}', q_{i}') > M$. Hence, $Len(V_{i}') \ge M$ for each $i \in I_{1}$. Summing up, we obtain $Len(P_{1}) \ge M \cdot \#I_{1}$.

Similar logic implies $Len(P_{2}) \ge M \cdot \# I_{2}$. We thus conclude  \[
Len(P_{1}) + Len(P_{2}) \ge M \cdot (\#I_{1} \cup \#I_{2}) \ge M \cdot k.\qedhere
\]
\end{proof}

\end{prop}

\section{Proof of Theorem \ref{thm:mainWPD}}\label{section:proof}
Throughout, let $G$ be a non-virtually cyclic group with a finite generating set $S$. Suppose that $G$ acts on a $\delta$-hyperbolic space $X \ni x_{0}$ with a WPD loxodromic element $\varphi \in G$. When a constant $L$ is understood, we will use the notation \[
\Upsilon_{L} := \big[x_{0}, \varphi^{L} x_{0} \big].
\]

Since $G$ contains independent loxodromics, $G$ has exponential growth. In other words, \[
\lambda_{S} := \liminf_{n} \frac{\ln \# B_{S}(n)}{n} > 1.
\]
This immediately implies that:

\begin{fact}\label{fact:expGrowth}
For each sufficiently large $n$ we have \[
\# B_{S}(0.9n) \big/ \#B_{S}(n) \le \lambda_{S}^{0.05n}.
\]
\end{fact}

Let us fix some more constants for the proof. Let $E_{0}$ be as in Fact \ref{fact:WPDNonElt}. Let $K_{Lip} := \max_{s \in S} \|s\|_{X}$, let $\tau := \lim_{n} \|\varphi^{n}\|_{X} / n$ (so that $\|\varphi^{k}\|_{X} \ge k \tau$ for each $k$). Let $F_{0} := \|\varphi\|_{S}$. Finally, let $L_{0}$ be as in Proposition \ref{prop:subContN} for $K = 100(E_{0} + 1000\delta + K_{Lip})$, and \[
L_{1} := L_{0} + \frac{1}{\tau} 100(E_{0} + 1000\delta + K_{Lip}).
\]
This choice implies that:
\begin{fact}\label{fact:constChoice}
For each $L > L_{1}$, \[
d_{X} (x_{0}, \varphi^{L} x_{0}) - 40(E_{0} + 1000\delta + K_{Lip}) \ge \tau L - 40(E_{0} + 1000\delta + K_{Lip})  \ge 0.5 \tau L +140\delta.
\]
\end{fact}

Given $L, \epsilon>0$, we define 
\[
\begin{aligned}
\mathcal{V}_{L, \epsilon}(n) &:= \left\{ g \in B_{S}(n)  :\begin{array}{c} \,\,\textrm{there exist $h_{1}, \ldots, h_{\epsilon n} \in G$ such that
} \\ \textrm{$(x_{0}, h_{1}\Upsilon_{L }, \ldots, h_{\epsilon n} \Upsilon_{L }, gx_{0})$ is $(6E_{0} + 300\delta)$-aligned}\end{array} \right\}, \\
\mathcal{BAD}_{L, \epsilon}(n) &:= B_{S}(n) \setminus \big( B_{S}(0.9n) \cup \mathcal{V}_{L, \epsilon}(n) \big).
\end{aligned}
\]

\begin{lem}\label{lem:bad}
For each $L > L_{1}$ and $\epsilon>0$, we have\[
\limsup_{n} \frac{\#\mathcal{BAD}_{L, \epsilon}(n) }{\#B_{S}(n)} < 5 \cdot (2E_{0} + 4L F_{0} + 5) \cdot  (\#S)^{E_{0} + 3L F_{0}+ 4} \cdot \epsilon.
\]\end{lem}

\begin{proof}[Proof of Lemma \ref{lem:bad}]
Let us define a map \[
F : Dom(F) := \mathcal{BAD}_{L , \epsilon}(n) \times \{1, \ldots, 0.9n\} \,\rightarrow \,B_{S}(n)
\] as follows. Given $(g, i) \in \mathcal{BAD}_{L , \epsilon}(n) \times \{1, \ldots, 0.9n\}$, we first fix a $d_{S}$-geodesic representative $g = a_{1} a_{2} \cdots a_{\|g\|_{S}}$. By Fact \ref{fact:WPDNonElt}, there exist $s=s(g, i)$ and $t=t(g,i)$ in $S \cup \{id\}$ such that \[
\big(s^{-1} \cdot (a_{1} \cdots a_{i})^{-1} x_{0}, \varphi^{L } x_{0}\big)_{x_{0}} < E_{0},
\big(\varphi^{-L } x_{0}, t \cdot a_{i + F_{0} L  + 3} \cdots a_{\|g\|} x_{0}\big)_{x_{0}} < E_{0}.
\]
We then define \[
h(g, i) := a_{1} \cdots a_{i} \cdot s, \,\,h'(g, i) := t\cdot a_{i + F_{0} L  + 3} \cdots a_{\|g\|}, \,\, F(g, i) := h(g, i) \varphi^{L } h'(g, i).
\]
Note that $F(g, i) \in B_{S}(n)$ because \[\begin{aligned}
\|F(g, i)\|_{S}& \le \|h(g, i)\|_{S} + \| \varphi^{L }\|_{S} + \|h'(g, i)\|_{S} \\
&\le (i+1) + F_{0}L  + (\|g\|_{S} - i - F_{0}L  - 2) + 1 \le \|g\|_{S} \le n.
\end{aligned}
\]

Before the proof, we first declare \[
T:=  (2E_{0} + 4L F_{0} + 5) \cdot  \# B_{S}(E_{0} + 2L F_{0}) \cdot \# B_{S}(F_{0} L  + 4)
\] 
and $R_{0} = R_{0}(L , 4/\epsilon)$ as in Proposition \ref{prop:subContN}.

\begin{claim}\label{claim:maxRed}
Let $m$ be the maximum number of elements $(g_{1}, i_{1}), \ldots, (g_{m}, i_{m}) \in \mathcal{BAD}_{L , \epsilon}(n) \times \{1, \ldots, 0.9n\}$ such that \begin{enumerate}
\item $F(g_{1}, i_{1}) = \ldots = F(g_{m}, i_{m}) =: U$,
\item $\Big(x_{0}, \,h(g_{1}, i_{1}) \cdot  \Upsilon_{L },\, \ldots, \, h(g_{m}, i_{m}) \cdot  \Upsilon_{L }, \,Ux_{0}\Big)$ is $6(E_{0} + 30\delta)$-aligned.
\end{enumerate}
Then $ \# F^{-1} (U) \le 2T \cdot m$ holds for each $U \in B_{S}(n)$.
\end{claim}

\begin{proof}[Proof of Claim \ref{claim:maxRed}]
Fix an arbitrary $U \in B_{S}(n)$. For each $(g, i) \in F^{-1}(U)$, we have:\begin{enumerate}
\item $(x_{0}, h(g, i) \varphi^{L } x_{0})_{h(g, i) x_{0}} < E_{0}$; hence $\pi_{h(g, i)\Upsilon_{L }}(x_{0})$ is $(E_{0}+8\delta)$-close to $h(g, i) x_{0}$ (Fact \ref{fact:thinTri}).
\item Similarly, the projection of $Ux_{0}$ onto $h(g, i)\Upsilon_{L }$ is $(E_{0}+8\delta)$-close to $h(g, i)\varphi^{L }x_{0}$. 
\item $h(g, i)x_{0}$ and $h(g, i)\varphi^{L }x_{0}$ are at least $\tau L $-apart, which is much larger than $20\delta$. 
\end{enumerate}
Now Fact \ref{fact:projectionCoarse} guarantees a subsegment $\gamma(g, i)$ of $[x_{0}, Ux_{0}]$ and a subsegment $\eta = [p, q]$ of $h(g, i) \Upsilon_{L }$ that are $20\delta$-fellow traveling. Here, $p$ and $h_{k}x_{0}$, and $q$ and $h_{k}\varphi^{L }x_{0}$ are pairwise $(E_{0} + 8\delta)$-close. Hence, $\gamma(g, i)$ and $h(g, i)\Upsilon_{L }$ are $(E_{0} + 30\delta)$-fellow traveling. It follows that $\gamma(g, i)$'s are longer than $\tau L  - 2(E_{0} + 30\delta) \ge 25 (E_{0} + 30\delta)$.

We now pick a maximal subset $\mathcal{A}$ of $F^{-1}(U)$ such that \[
\textrm{for any $(g, i), (g', i') \in \mathcal{A}$, $\gamma(g, i)$ and $\gamma(g', i')$ overlap for length at most $12(E_{0} + 30\delta)$}. 
\]
We claim that $\# F^{-1}(U) \le T \cdot \#\mathcal{A}$. To show this, pick an arbitrary $(g, i) \in F^{-1}(U)$. Let $a_{1} \cdots a_{\|g\|_{S}}$ be the geodesic representative for $g$ that was used when defining\[
h(g, i)  := a_{1} \cdots a_{i} \cdot s(g, i), \quad h(g, i)' := t(g, i) \cdot a_{i + F_{0}L  + 3} + \cdots a_{\|g\|_{S}}.
\]
By the maximality of $\mathcal{A}$, there exists $(\mathfrak{g}, \mathfrak{i}) \in \mathcal{A}$ such that $\gamma(g, i)$ and $\gamma(\mathfrak{g}, \mathfrak{i})$ overlap for length at least $12(E_{0} + 30\delta)$. Recall that $h(g, i) \Upsilon_{L }$ and $h(\mathfrak{g}, \mathfrak{i})\Upsilon_{L }$ are $(E_{0} + 30\delta)$-fellow traveling $\gamma(g, i)$ and $\gamma(\mathfrak{g}, \mathfrak{i})$, respectively. By Fact \ref{fact:fellowOverlap}, $\pi_{h(g, i)\Upsilon_{L }} (h(\mathfrak{g}, \mathfrak{i}) x_{0})$ appears earlier than $\pi_{h(g, i)\Upsilon_{L }} (h(\mathfrak{g}, \mathfrak{i}) \varphi^{L }x_{0})$. Moreover, they are $(12(E_{0} + 30\delta) - 10(E_{0} + 30\delta))$-apart and  hence $E_{0}$-apart. By Fact \ref{fact:WPDNonElt}, we conclude that $\varphi^{-a} \cdot h(g,i)^{-1} h(\mathfrak{g}, \mathfrak{i}) \varphi^{b} \in B_{S}(E_{0})$ for some $a, b \in \{0, \ldots, L \}$. We conclude that \[
h(g, i) \in h(\mathfrak{g}, \mathfrak{i}) \cdot \{\varphi^{a} : a = 0, \ldots, L \} \cdot B_{S}(E_{0}) \cdot \{\varphi^{-a} : a = 0, \ldots, L \} \subseteq  h(\mathfrak{g}, \mathfrak{i})  B_{S}(E_{0} + 2L F_{0}).
\]
This also implies that $\|h(g, i)\|_{S}$ and $\|h(\mathfrak{g}, \mathfrak{i})\|_{S}$ differ by at most $E_{0} + 2 L F_{0}$, and hence \[
i \in \big[\mathfrak{i} - (E_{0} + 2 L F_{0} + 2), \mathfrak{i} + (E_{0} + 2 L F_{0}+2)\big].
\]
Note also that
\[
h(g, i)' = \varphi^{-L } h(g, i)^{-1} \cdot U
\]
is determined as soon as $h(g, i)$ is determined.

Finally, in order to reconstruct $g = a_{1} \cdots a_{\|g_{l}\|_{S}}$ from $h(g, i)$ and $h(g, i)$, it suffices to pick $c := s_{l}^{-1} a_{i_{l}+1} \cdots a_{i_{l} + F_{0}L  + 2} t^{-1} \in B_{S}(F_{0} L  + 4)$ and multiply $h(g, i)$, $c$ and $h(g', i')$.  In summary, we have\[\begin{aligned}
&F^{-1}(U)  &\subseteq \bigcup_{(\mathfrak{g}, \mathfrak{i}) \in \mathcal{A}} \bigg( \Big\{ h(\mathfrak{g}, \mathfrak{i})  f \cdot c \cdot \varphi^{-L } f^{-1} h(\mathfrak{g}, \mathfrak{i})^{-1}  U : f \in B_{S}(E_{0} + 2L F_{0}), \,c \in B_{S}(F_{0}L  + 4) \Big\} \times I(\mathfrak{i}) \bigg)
\end{aligned}
\]
where $I(\mathfrak{i}):= [\mathfrak{i} - (E_{0} + 2 L F_{0} + 2), \mathfrak{i} + (E_{0} + 2 L F_{0} + 2)]$. From this, we conclude $\#F^{-1}(U) \le T \cdot \#\mathcal{A}$.

Next, let us enumerate $\mathcal{A}$ as \[
\mathcal{A} = \{ (\mathfrak{g}_{1}, \mathfrak{i}_{1}), (\mathfrak{g}_{2},  \mathfrak{i}_{2}), \ldots\}
\]
so that $\gamma(\mathfrak{g}_{l},   \mathfrak{i}_{l})$ starts earlier than $\gamma(\mathfrak{g}_{l+1},  \mathfrak{i}_{l+1})$ along $[x_{0}, Ux_{0}]$, for each $l$. Then the beginning point of $\gamma(\mathfrak{g}_{2},   \mathfrak{i}_{2})$ is later than that of $\gamma(\mathfrak{g}_{1},  \mathfrak{i}_{1})$ and earlier than that of $\gamma(\mathfrak{g}_{3},   \mathfrak{i}_{3})$. ($\ast$) Moreover, $\gamma(\mathfrak{g}_{2},   \mathfrak{i}_{2})$ does not contain $\gamma(\mathfrak{g}_{3},   \mathfrak{i}_{3})$, as their overlap should not be longer than $12(E_{0}+30\delta)$ while $\gamma(\mathfrak{g}_{3},   \mathfrak{i}_{3})$ is longer than $25(E_{0} + 30\delta)$. Hence, the ending point of $\gamma(\mathfrak{g}_{2},   \mathfrak{i}_{2})$ is earlier than that of $\gamma(\mathfrak{g}_{3},   \mathfrak{i}_{3})$. ($\ast\ast$)

At this point, if $\gamma(\mathfrak{g}_{1},   \mathfrak{i}_{1})$ and $\gamma(\mathfrak{g}_{3},   \mathfrak{i}_{3})$ intersect, then $\gamma(\mathfrak{g}_{3},   \mathfrak{i}_{3})$ is completely covered by $\gamma(\mathfrak{g}_{1},   \mathfrak{i}_{1})$ and $\gamma(\mathfrak{g}_{3},   \mathfrak{i}_{3})$  due to $(\ast)$ and $(\ast \ast)$. We would then have \[
\diam_{X}\big(\gamma(\mathfrak{g}_{1},   \mathfrak{i}_{1}) \cap \gamma(\mathfrak{g}_{2},   \mathfrak{i}_{2}) \big) + 
\diam_{X}\big(\gamma(\mathfrak{g}_{2},   \mathfrak{i}_{2})  \cap \gamma(\mathfrak{g}_{3},   \mathfrak{i}_{3}) \big)  \ge \diam_{X}\big(\gamma(\mathfrak{g}_{2},   \mathfrak{i}_{2})\big) \ge 25(E_{0} +30\delta),
\]
which contradicts to the bound $12(E_{0} + 30\delta)$ on $\diam_{X}\big(\gamma(\mathfrak{g}_{1},   \mathfrak{i}_{1}) \cap \gamma(\mathfrak{g}_{2},   \mathfrak{i}_{2}) \big)$ and $\diam_{X}\big(\gamma(\mathfrak{g}_{2},   \mathfrak{i}_{2})  \cap \gamma(\mathfrak{g}_{3},   \mathfrak{i}_{3}) \big)$. Hence, we conclude that $\gamma(\mathfrak{g}_{1},   \mathfrak{i}_{1})$ and $\gamma(\mathfrak{g}_{3},   \mathfrak{i}_{3})$ do not intersect. 

With the same logic, we conclude that $\gamma(\mathfrak{g}_{l}, \mathfrak{i}_{l})$'s for odd integers $l$ are disjoint subsegments of $[x_{0}, Ux_{0}]$,  in order from left to right along $[x_{0}, Ux_{0}]$. Recall again that $\gamma(\mathfrak{g}_{l}, \mathfrak{i}_{l})$ and $h(\mathfrak{g}_{l}, \mathfrak{i}_{l})) \Upsilon_{L }$ are $(E_{0} + 30\delta)$-fellow traveling. Now Fact \ref{fact:fellowAlign} tells us that \[
\big(x_{0}, h(\mathfrak{g}_{1}, \mathfrak{i}_{1})\Upsilon_{L }, \,h(\mathfrak{g}_{3}, \mathfrak{i}_{3}) \Upsilon_{L }, \, \ldots, \,h(\mathfrak{g}_{2\lceil\#\mathcal{A}/2\rceil + 1}, \mathfrak{i}_{2\lceil\#\mathcal{A}/2\rceil + 1})\Upsilon_{L }, \,\,Ux_{0}\big)
\]
is $6(E_{0} + 30\delta)$-aligned. This implies that $m \ge \lceil\#\mathcal{A}/2\rceil  \ge \frac{1}{2T} \#F^{-1} (U)$ as desired.
\end{proof}

Now let $(g_{1}, i_{1}), \ldots, (g_{m}, i_{m}) \in \mathcal{BAD}_{L , \epsilon}(n) \times \{1, \ldots, 0.9n\}$ the elements as in Claim \ref{claim:maxRed}: \begin{enumerate}
\item $F(g_{1}, i_{1}) = \ldots = F(g_{m}, i_{m}) =: U$,
\item $\Big(x_{0}, \,h(g_{1}, i_{1}) \cdot  \Upsilon_{L },\, \ldots, \, h(g_{m}, i_{m}) \cdot  \Upsilon_{L }, \,Ux_{0}\Big)$ is $6(E_{0} + 30\delta)$-aligned.
\end{enumerate}

It remains to prove that $m < 2\epsilon n$ for large enough $n$. We will prove it for every $n \ge 32R_{0} K_{Lip} / \tau L  \epsilon$. Suppose to the contrary that $m \ge 2\epsilon n$. Let us denote the $d_{S}$-geodesic representative used for $g_{1}$ by  $a_{1}\cdots a_{\|g_{1}\|_{S}}$, so that $h(g_{1}, i_{1}) = a_{1} \cdots a_{i_{1}} s(g_{1}, i_{1})$. We will abbreviate $h(g_{l}, i_{l})$ by $h_{l}$, $h'(g_{l}, i_{l})$ by $h_{l}'$, $s(g_{l}, i_{l})$ by $s_{l}$ and $t(g_{l}, i_{l})$ by $t_{l}$.

We focus on a particular vertex on the $d_{S}$-geodesic $Ug_{1}^{-1} \cdot [x_{0}, g_{1}]_{S}$, namely \[
v := Ug_{1}^{-1} \cdot a_{1} \cdots a_{i_{1} + F_{0}L  + 2} = a_{1} \cdots a_{i_{1}} \cdot s_{1}\cdot \varphi^{L } \cdot t_{1}x_{0} = h_{1} \cdot \varphi^{L } \cdot t_{1}.
\]
Then $vx_{0}$ is $K_{Lip}$-close to $h_{1} \cdot \varphi^{L } x_{0}$. Since $\big(h_{1}\varphi^{L} x_{0}, \,h_{2}[x_{0}, \varphi^{L}x_{0}]\big)$ is $(6E_{0} + 180\delta)$-aligned, Fact \ref{fact:projectionCoarse}(1) tells us that $\big(vx_{0}, \,h_{2}\Upsilon_{L }\big)$ is $(6E_{0} + 200\delta + K_{Lip})$-aligned.

Now, Fact \ref{fact:Behr} tells us that either: \begin{enumerate}
\item $(Ug_{1}^{-1} x_{0}, h_{\epsilon n} \Upsilon_{L })$ is $(6E_{0} + 240\delta)$-aligned,  or 
\item $(h_{\epsilon n - 1} \Upsilon_{L }, Ug_{1}^{-1} x_{0})$ is $(6E_{0} + 240\delta)$-aligned. 
\end{enumerate} In Case (1), we conclude that \[
\big(x_{0}, \,g_{1} U^{-1} h_{\epsilon n} \Upsilon_{L }, \,g_{1} U^{-1} h_{\epsilon n+1} \Upsilon_{L },\, \ldots,\, g_{1} U^{-1} h_{m} \Upsilon_{L },\, g_{1} x_{0} \big)
\]
is $(6E_{0} + 240\delta)$-aligned. This contradicts the fact that $g_{1} \notin \mathcal{V}_{L , \epsilon}(n)$.

In the latter case, we have: \[
(vx_{0}, h_{1} \Upsilon_{L }, \ldots, h_{\epsilon n -1} \Upsilon_{L }, y_{i})
\]
is $(6E_{0} + 240\delta + K_{Lip})$-aligned for $y_{1} = U g_{1}^{-1} x_{0}$ and $y_{2} = Ux_{0}$. Here, the alignment of $(h_{\epsilon n-1} \Upsilon_{L }, y_{2})$ is due to Fact \ref{fact:Behrstock}. Let $P_{1}$ be the first half of the geodesic $Ug_{1}^{-1} [x_{0}, g_{1} x_{0}]$ connecting $Ug_{1}^{-1} x_{0}$ to $vx_{0}$, and let $P_{2}$ be the latter half connecting $vx_{0}$ to $Ux_{0}$. Then $Len(P_{1}) + Len(P_{2}) \le \|g_{1}\|_{S} \le n$.

Recall that $R_{0} = R_{0}(L , 4/\epsilon)$ is chosen as in Proposition \ref{prop:subContN} and that $L  \ge L_{1}$ is longer than $L_{0}(K)$ for $K=6E_{0} + 240\delta + K_{Lip}$. Since $Len(P_{1}) + Len(P_{2}) \le n \le (4/\epsilon) \cdot (\epsilon n/4)$, the paths should satisfy the first alternative in Proposition \ref{prop:subContN} for $k=\epsilon n/4$. In particular, there exists $i \in  \{0.5\epsilon n, \ldots, 0.75 \epsilon n\}$ such that $d_{S}(P_{1}, h_{i}), d_{S}(P_{2}, h_{i}) \le R_{0}$. Let $u_{1} \in P_{1}$ and $u_{2} \in P_{2}$ be the vertices realizing the distance.

Meanwhile, note that $(vx_{0}, h_{2} \Upsilon_{L }, \ldots, h_{i-1} \Upsilon_{L }, h_{i} x_{0})$ is $(6E_{0} + 200\delta + K_{Lip})$-aligned. Fact \ref{fact:GromProdFellow} implies that 
 there exist $i-2 \ge 0.25\epsilon n$ disjoint subsegments of $[vx_{0}, h_{i}x_{0}]$, each longer than $\tau L  - 2(6E_{0} + 200\delta + K_{Lip}) - 160\delta \ge 0.5\tau L $. This implies that \[
d_{S} (h_{i}, v) \ge \frac{1}{K_{Lip}} d_{X}(vx_{0}, h_{i} x_{0}) \ge\frac{1}{K_{Lip}} \cdot 0.5 \tau L \cdot 0.35 \epsilon n.
\]
This implies that \[
d_{S}(u_{1}, v) \ge d_{S}(h_{i}, v) - d_{S} (h_{i}, u_{1}) \ge \frac{\tau L  \epsilon n}{8K_{Lip}}  - R_{0} \ge 3R_{0}. 
\]
Meanwhile, $u_{1}, v$ and $u_{2}$ are aligned along a $d_{S}$-geodesic $U g_{1}^{-1} [x_{0}, g_{1} x_{0}]$. This leads to a contradiction \[
2R_{0} \ge d_{S}(u_{1}, h_{i}) + d_{S}(u_{2}, h_{i}) \ge  d_{S}(u_{1}, u_{2}) \ge d_{S}(u_{1}, v) \ge 3R_{0}
\]

In conclusion, $m \le 2 \epsilon n$ holds for $m$ in Claim \ref{claim:maxRed} when $n \ge 32R_{0} K_{Lip} / \tau L  \epsilon$. This implies that \[
\# \mathcal{BAD}_{L , \epsilon}(n) \times 0.9n = \# \big(\operatorname{Dom} F \big)= \sum_{U \in B_{S}(n)} \big(\# F^{-1}(U) \big) \le 4T \epsilon n \cdot \# B_{S}(n).
\]
We conclude  \[
\frac{\# \mathcal{BAD}_{L , \epsilon}(n)}{\#B_{S}(n)} \le 5T \epsilon. \quad (\forall n \ge 32R_{0} K_{Lip} / \tau L  \epsilon) \qedhere
\]
 \end{proof}
 
Let us now define  \[
\mathcal{W}_{L , \epsilon}(n) := \left\{ g \in B_{S}(n)  :\begin{array}{c} \,\,\exists h_{1}, \ldots, h_{\epsilon n} \in G\,\,\textrm{such that the sequences $(x_{0}, h_{1}\Upsilon_{L }, \ldots, h_{\epsilon n} \Upsilon_{L }, gx_{0})$,}\\
\textrm{$(g^{-1} x_{0}, h_{1}\Upsilon_{L })$ and $(h_{\epsilon n}\Upsilon_{L }, g^{2} x_{0})$ are each $(6E_{0} + 360\delta)$-aligned}
\end{array} \right\}.
\]

\begin{lem}\label{lem:WUseful}
Let $L  > L_{1}$. Then the following is true for $g \in \mathcal{W}_{L , \epsilon}(n)$:\begin{enumerate}
\item $g$ is a loxodromic isometry on $X$ with $\tau_{X}(g) \ge 0.5\tau L  \epsilon n$.
\item $g$ has the WPD property and hence is Morse (\cite[Theorem 1]{sisto2016quasi-convexity}).
\item There exists a conjugate $\psi$ of $\varphi$ such that for each large enough $i$, the projections of $g^{-i} x_{0}$ and $g^{i}x_{0}$ onto $[\psi^{-i}x_{0}, \psi^{i} x_{0}]$ are at least $\tau L /2$-apart, with the former one coming first.
\end{enumerate}
\end{lem}

\begin{proof}
For the first item, we claim that \begin{equation}\label{eqn:gromSeq}
\big( \ldots, g^{-1} \gamma_{1}, \ldots, g^{-1} \gamma_{\epsilon n}, \gamma_{1}, \ldots, \gamma_{\epsilon n}, g \gamma_{1}, \ldots, g \gamma_{\epsilon n}, \ldots \big)
\end{equation}
is $(12E_{0} + 900\delta)$-aligned, where $\gamma_{i} := h_{i} \Upsilon_{L }$. The only nontrivial part is the $(12E_{0} + 900\delta)$-alignment of $(\gamma_{\epsilon n}, g\gamma_{1})$. First, observe that $(\gamma_{\epsilon n}, gx_{0})$ and $(\gamma_{\epsilon n}, g^{2}x_{0})$ are each $(6E_{0} + 360\delta)$-aligned. By Corollary \ref{cor:monotone}, $(\gamma_{\epsilon n}, z)$ is $(6E_{0} + 420\delta)$-aligned for each $z \in [gx_{0}, g^{2} x_{0}]$.

Meanwhile, Fact \ref{fact:GromProdFellow} implies that $g\gamma_{1}$ is contained in the $(6E_{0} + 440\delta)$-neighborhood of $[x_{0}, gx_{0}]$. Fact \ref{fact:projectionCoarse}(1) implies that $\pi_{\gamma_{\epsilon n}} (g \gamma_{1})$ is contained in the $(12E_{0} + 900\delta)$-long ending subsegment of $\gamma_{\epsilon n}$.

By a symmetric argument, we can similarly observe that $\pi_{g\gamma_{1}} (\gamma_{\epsilon n})$ is contained in the $(12E_{0} + 900\delta)$-long ending subsegment of $g\gamma_{1}$. This concludes the desired alignment.

Now Fact \ref{fact:GromProdFellow} applies to the $(12E_{0} + 900\delta)$-aligned sequence $(x_{0}, \gamma_{1}, \ldots, \gamma_{\epsilon n}, g \gamma_{1}, \ldots, g \gamma_{\epsilon n}, \ldots, g^{k} x_{0})$ 
and concludes that \[\begin{aligned}
d_{X}(x_{0}, g^{k} x_{0}) &\ge \sum_{i=0}^{k-1} \sum_{j =1}^{\epsilon n} \big( \diam_{X}(g^{i} \gamma_{j}) - (2(12E_{0} + 900\delta) + 160\delta) \big) \\
&\ge \epsilon n k \cdot \big(\tau L  -  (2(12E_{0} + 900\delta)+ 160\delta) \big) \ge  \epsilon n k \cdot \frac{1}{2}\tau L .
\end{aligned}
\]
This implies that $\tau_{X}(g) \ge 0.5 \tau L  \epsilon n$.

In order to discuss WPD property, let $K >0$. Because $g$ is loxodromic, there exists $N$ such that $d_{X}(g^{\pm N} x_{0}, h_{1}\Upsilon_{L }) \ge K + 1000\delta$. We then claim that $Stab_{K}(x_{0}, g^{2N} x_{0})$ is finite. Suppose to the contrary that $Stab_{K}(x_{0}, g^{2N} x_{0})$ is not contained in any finite $d_{S}$-metric ball. Then we can take infinitely many distinct elements $g_{1}, g_{2}, \ldots \in Stab_{K}(x_{0}, g^{2N} x_{0})$.

Combining the alignment of the sequence in Display \ref{eqn:gromSeq} and Fact \ref{fact:Behrstock}, we observe that \[(x_{0}, g^{N} h_{1} \Upsilon_{L }, g^{2N} x_{0})
\] is $(12E_{0} + 960\delta)$-aligned. Since $d_{X}(x_{0}, g_{i}x_{0} ) \le K \le  d_{X} (x_{0}, g^{N} h_{1} \Upsilon_{L }) - 1000\delta$, the contraposition of Fact \ref{fact:projectionCoarse}(2) tells us that $\pi_{g^{N} h_{1} \Upsilon_{L }}(\{x_{0}, g_{i} x_{0}\})$ has diameter at most $20\delta$. Similarly, $\pi_{g^{N} h_{1} \Upsilon_{L }}(\{g^{2N}_{0}, g_{i} g^{2N} x_{0}\})$ is also $20\delta$-small. Hence, $(g_{i} x_{0}, g^{N} h_{1} \Upsilon_{L }, g_{i} g^{2N} x_{0})$ is also $(12E_{0} + 980\delta)$-aligned. Hence, $[g_{i} x_{0}, g_{i} g^{2N} x_{0}]$ contains a subsegment $\eta_{i}$ that is $(12E_{0} + 1060\delta)$-fellow traveling with $g^{N} h_{1} \Upsilon_{L }$.

Now, $g_{i}^{-1} \eta_{i}$'s are subsegments of $[x_{0}, g^{2N} x_{0}]$ that is longer than $\tau L  - 2(12E_{0} + 1060\delta) \ge 0.5\tau L $. Since $[x_{0}, g^{2N} x_{0}]$ is compact, by passing to subsequence, we may assume that $g_{i}^{-1} \eta_{i}$'s converge to a subsegment of $[x_{0}, g^{2N} x_{0}]$ of length at least $0.5 \tau L $. Also, these subsegments are $(12E_{0} + 1060\delta)$-fellow traveling with $g_{i}^{-1} g^{N} h_{1} \Upsilon_{L }$ and $g_{j}^{-1} g^{N} h_{1} \Upsilon_{L }$, respectively. Since $0.5\tau L  > 12 (12 E_{0} + 1060 \delta) + E_{0}$,  Fact \ref{fact:fellowOverlap} implies for large $i, j$ that $\pi_{g_{i}^{-1} g^{N} h_{1} \Upsilon_{L }} (g_{j}^{-1} g^{N} h_{1} \Upsilon_{L })$ is $E_{0}$-large and is orientation-matching. Now Fact \ref{fact:WPDNonElt}(2) implies that \[
h_{1}^{-1} g^{-N} g_{i} g_{j}^{-1} g^{N} h_{1} \subseteq B_{S}(E_{0} + 2L F_{0}).
\]
In particular, $g_{i} g_{j}^{-1}$ is uniformly bounded for every pair of $g_{i}, g_{j}$. This contradicts the infinitude of $Stab_{K}(x_{0}, g^{2N} x_{0})$. The WPD property of $g$ is now proven.

The third item holds for $\psi = h_{1} \varphi h_{1}^{-1}$. Indeed, when $N$ is sufficiently large, $[g^{-N} x_{0}, g^{N} x_{0}]$ and $[ h_{1} \varphi^{-N} h_{1}^{-1}  x_{0}, h_{1} \varphi^{N} h_{1} x_{0}]$ both contain subsegments that are $0.01\tau L $-fellow traveling with a $\tau L $-long geodesic $h_{1} \Upsilon_{L }$. We omit the detail.
\end{proof}

We now claim that $\mathcal{V}_{L , 3\epsilon}(n) \setminus \mathcal{W}_{L , \epsilon}(n)$ is non-generic. 

\begin{lem}\label{lem:VMinusW}
For each $L  > L_{1}$ and $\epsilon > 0$, there exists $\lambda > 1$ such that\[
\lim_{n\rightarrow +\infty} \frac{\# \mathcal{V}_{L , 3\epsilon}(n) \setminus \mathcal{W}_{L , \epsilon}(n)}{\#B_{S}(n)} \le \lambda^{-n}
\]
for all large enough $n$.
\end{lem}

\begin{proof}
Before the proof, let $R_{1} = R_{1}(L , 12/\epsilon)$ be as in Proposition \ref{prop:subContN}. Let us first define \[
\begin{aligned}
\mathcal{K}_{1} &:= \bigcup_{r =\epsilon n}^{n/2} \big\{ abc a^{-1} : a \in B_{S}(r), b \in B_{S}( n - 2r + 2R_{1}), c \in B_{S}(2R_{1})\big\}, \\
\mathcal{K}_{2} &:= \bigcup_{r, r'\ge 0, r+r' \le (1-\epsilon)n}\big\{ acba^{-1} : a \in B_{S}(r + 2R_{1}), b \in B_{S} (r' + 2R_{1}), c \in B_{S}(2R_{1})\big\}.
\end{aligned}
\]
We also define $\mathcal{K}_{i}^{-1} := \{g^{-1} : g \in \mathcal{K}_{i}\}$ for $i=1, 2$. Then we have 
\[
\# \big(\mathcal{K}_{1}  \cup \mathcal{K}_{2}  \cup \mathcal{K}_{1}^{-1}  \cup \mathcal{K}_{2}^{-1} \big) \lesssim 2 \cdot n^{2} \lambda_{S}^{(1-0.5\epsilon)n}.
\]
This is exponentially smaller than $\#B_{S}(n)$. 

It remains to prove that $\mathcal{V}_{L, 3\epsilon}(n) \setminus \big(\mathcal{K}_{1}  \cup \mathcal{K}_{2}  \cup \mathcal{K}_{1}^{-1}  \cup \mathcal{K}_{2}^{-1} \big)$ is contained in $\mathcal{W}_{L, \epsilon}(n)$. To show this, let $g \in \mathcal{V}_{L , 3\epsilon}(n) \setminus \big(\mathcal{K}_{1}  \cup \mathcal{K}_{2}  \cup \mathcal{K}_{1}^{-1}  \cup \mathcal{K}_{2}^{-1} \big) $. Then there exists $h_{1}, \ldots, h_{3\epsilon n} \in G$ such that \begin{equation}\label{eqn:hiAlignEx}
(x_{0}, h_{1} \Upsilon_{L }, \ldots, h_{3\epsilon n} \Upsilon_{L }, gx_{0})
\end{equation}
is $(6E_{0} + 300\delta)$-aligned. Then $(x_{0}, h_{i}  \Upsilon_{L }, gx_{0})$ is $(6E_{0} + 360\delta)$-aligned for each $i$ by Fact \ref{fact:Behrstock}.

Fact \ref{fact:Behr} guarantees that the following dichotomy holds: either\begin{enumerate}
\item $(g^{-1} x_{0}, h_{\epsilon n + 1} \Upsilon_{L })$ is $(6E_{0} + 360\delta)$-aligned, or 
\item $(h_{\epsilon n}  \Upsilon_{L }, g^{-1} x_{0})$ is $(6E_{0} + 360\delta)$-aligned.
\end{enumerate}
We claim that Case (1) holds. Suppose to the contrary that Case (2) holds. That means, \[
(x_{0}, h_{1} \Upsilon_{L }, \ldots, h_{\epsilon n} \Upsilon_{L },  g^{\pm 1} x_{0})\,\,\textrm{is $(6E_{0} + 360\delta)$-aligned}.
\]
Pick a $d_{S}$-geodesic path $P$ connecting $id$ to $g$. Then $g^{-1} P$ is a path connecting $g^{-1}$ to $id$. 

We now apply Proposition \ref{prop:subContN}. Since $Len(P) + Len(g^{-1} P) \le 2n \le (12/\epsilon) \cdot (\epsilon n /6)$, the first alternative in Proposition \ref{prop:subContN} should hold for $k = \epsilon n/6$. In particular, there exists $i \in \{\epsilon n / 2, \ldots, \epsilon n\}$ such that $d_{S}(h_{i}, P), d_{S}(h_{i}, g^{-1} P) \le R_{1}$. Let $v \in P$ and $g^{-1} u \in g^{-1} P$ be the vertices realizing the distance. Here, as before, the alignment of the sequence in Display \ref{eqn:hiAlignEx} implies that $[x_{0}, h_{i} x_{0}]$ contains $i$ disjoint subsegments longer than $\diam_{X}(\Upsilon_{L }) - 2(6E_{0} + 520\delta) \ge 0.5\tau L$. Hence, \[
d_{S}(id, h_{i}) \ge \frac{1}{K_{Lip}} d_{X} (x_{0}, h_{i} x_{0}) \ge \frac{1}{4K_{Lip}} \tau L  \epsilon n.
\]
This implies  \[
\|v\|_{S} \ge \|h_{i}\|_{S} - d_{S}(h_{i}, v) \ge \frac{1}{4K_{Lip}} \tau L  \epsilon n - R_{1} \ge \epsilon n. \quad(\textrm{when $n \ge R_{1}/\epsilon$})
\]
Meanwhile, since $(h_{i} x_{0}, h_{i+1} \Upsilon_{L }, \ldots, h_{3 \epsilon n} \Upsilon_{L }, gx_{0})$ is also aligned, we have \[
d_{S}(h_{i}, g) \ge \frac{1}{K_{Lip}} d_{X} (h_{i} x_{0}, gx_{0}) \ge \frac{1}{K_{Lip}} \tau L  \epsilon.
\]
This implies $d_{S}(v, g) \ge \epsilon n$. Note that $\|g^{-1} u\|_{S} = \|g\|_{S} - \|u\|_{S}$ and $\|v\|_{S}$ differ by at most $2R_{0}$. ($\ast$)

We now divide the cases:
\begin{enumerate}
\item $\epsilon n \le \|v\|_{S} \le \|g\|_{S}/2$. Recall that $id, u, v, g$ are on the same $d_{S}$-geodesic $P$.  This means\[
\|v^{-1} u\|_{S} = d_{S}(v, u) = \big| \|v\|_{S} - \|u\|_{S} \big| = \Big| \|v\|_{S} + \big( \|g^{-1} u\|_{S} - \|g\|_{S} \big)  \Big| .
\]
Thanks to $(\ast)$, we have \[
\Big| \|v\|_{S} + \big( \|g^{-1} u\|_{S} - \|g\|_{S} \big)  \Big| \le \big| 2\|v\|_{S} - \| g\|_{S} \big| + 2R_{0}  = \|g\|_{S} - 2\|v\|_{S} + 2R_{0}.
\]
Finally, $g^{-1} u$ and $v$ are $2R_{0}$-close so $u^{-1} g \cdot v \in B_{S}(2R_{0})$. This implies the contradiction \[
g = v \cdot (v^{-1} u) \cdot (u^{-1} g v) \cdot v^{-1} \in \mathcal{K}_{1}.
\]
\item $\|g\|_{S}/2 \le \|v\|_{S} \le \|g\|_{S} - \epsilon n$. In this case, $(\ast)$ implies that \[
\|u\|_{S} \le \|g\|_{S} - \|v\|_{S} + 2R_{0},\quad  \|u^{-1} v\|_{S} \le  \big| 2\|v\|_{S} - \| g\|_{S} \big| + 2R_{0} = 2\|v\|_{S} - \|g\|_{S} + 2R_{0}.
\]
We also have $u^{-1} g \cdot v^{-1} \in B_{S}(2R_{0})$. Note that $\|g\|_{S} - \|v\|_{S}$, $2\|v\|_{S} - \|g\|_{S}$ are positive integers whose sum is at most $\|v\|_{S} \le \|g\|_{S} - \epsilon n \le n - \epsilon n$.  These facts lead to a contradiction  \[
g = u \cdot (u^{-1} g v) \cdot (v^{-1} u) \cdot u^{-1} \in \mathcal{K}_{2}.
\]
\end{enumerate}

We can thus conclude that Case (1) holds. Meanwhile, Fact \ref{fact:Behr} asserts that either \begin{enumerate}[label=(\alph*)]
\item $( h_{2 \epsilon n} \Upsilon_{L }, g^{2} x_{0})$ is $(6E_{0} + 360\delta)$-aligned, or 
\item $(g^{2} x_{0},  h_{2 \epsilon n + 1} \Upsilon_{L })$ is $(6E_{0} + 360\delta)$-aligned.
\end{enumerate}
In Case (b), we are led to the alignment that
 \[
(g^{\pm 1} x_{0}, g^{-1}h_{2\epsilon n + 1} \Upsilon_{L }, \ldots, g^{-1} h_{3\epsilon n} \Upsilon_{L },  x_{0})\,\,\textrm{is $(6E_{0} + 360\delta)$-aligned}.
\]
A similar argument as before implies $g \in \mathcal{K}_{1}^{-1} \cup \mathcal{K}_{2}^{-1}$, a contradiction. Hence, Case (a) must hold.

In conclusion, the following sequence is  $(6E_{0} + 360\delta)$-aligned: \[
(g^{-1} x_{0},  h_{\epsilon n + 1} \Upsilon_{L },  \ldots, h_{2 \epsilon n} \Upsilon_{L }, g^{2} x_{0})
\]
Also,  $(x_{0}, h_{\epsilon n + 1} \Upsilon_{L })$ and $(h_{2\epsilon n} \Upsilon_{L }, gx_{0})$ are  $(6E_{0} + 360\delta)$-aligned. Hence $g \in \mathcal{W}_{L, \epsilon}(n)$.
\end{proof}

We can now finish the proof of Theorem \ref{thm:mainGen}.

\begin{proof}[Proof of Theorem \ref{thm:mainGen}]
We again start by fixing the constants $E_{0}$, $\tau$, $K_{Lip}, F_{0}, L_{1}$. Take $L\ge L_{1}$ large enough such that $\tau L \ge M$.

By Lemma \ref{lem:WUseful}, it suffices to show that for each $\eta>0$ there exists $\epsilon>0$ such that \begin{equation}\label{eqn:target}
\limsup_{n \rightarrow +\infty} \frac{\#\big(B_{S}(n) \setminus \mathcal{W}_{L, \epsilon}(n)\big)}{\#B_{S}(n)} \le \eta.
\end{equation}
To this end, we take \[
\epsilon := \frac{1}{30 (2E_{0} + 4L F_{0} + 5) (\#S)^{E_{0} + 3 L F_{0} + 4}} \cdot \eta.
\] Then by Fact \ref{fact:expGrowth}, Lemma \ref{lem:bad} and Lemma \ref{lem:VMinusW}, \[
\begin{aligned}
\lim_{n \rightarrow +\infty} &\frac{\#B_{S}(0.9n)}{\#B_{S}(n)} = \lim_{n \rightarrow +\infty} \frac{\#\big(\mathcal{V}_{L, 3\epsilon}(n) \setminus \mathcal{W}_{L, \epsilon}(n)\big)}{\#B_{S}(n)}&= 0,\\
\limsup_{n \rightarrow +\infty} &\frac{\#\mathcal{BAD}_{L, 3\epsilon}(n)}{\#B_{S}(n)} < \eta/2.\\
\end{aligned}
\]
Moreover, we have \[\begin{aligned}
B_{S}(n) \setminus \mathcal{W}_{L, \epsilon}(n) &\subseteq B_{S}(0.9n) \cup \Big(\big(B_{S}(n) \setminus B_{S}(0.9n)\big)\setminus  \mathcal{W}_{L, \epsilon}(n)\Big) \\
&\subseteq B_{S}(0.9n) \cup \Big(B_{S}(n) \setminus \big(B_{S}(0.9n) \cup \mathcal{V}_{L, 3\epsilon}(n) \big)\Big) \cup \big(\mathcal{V}_{L, 3\epsilon}(n) \setminus  \mathcal{W}_{L, \epsilon}(n) \big) \\
&= B_{S}(0.9n) \cup \mathcal{BAD}_{L, 3\epsilon}(n) \cup \big(\mathcal{V}_{L, 3\epsilon}(n) \setminus  \mathcal{W}_{L, \epsilon}(n) \big).
\end{aligned}
\]
Hence, Equation \ref{eqn:target} holds. 
\end{proof}

\begin{proof}[Proof of Theorem \ref{thm:mainGen}]
We only list additional observations needed for Theorem \ref{thm:mainGen}. For detailed explanations about the notion of principal/triangular/ageometric fully irreducible outer automorphism in $\Out(F_{n})$, refer to \cite{algom-kfir2019stable} and \cite{kapovich2022random}.

By \cite[Example 6.1]{algom-kfir2019stable}, there exists a principal fully irreducible $\varphi \in \Out(F_{N})$. Now \cite[Remark 5.4]{kapovich2022random} provides a \emph{lone axis} $\gamma$ for $\varphi$, which is necessarily a periodic greedy folding line. Further, every fully irreducibe $g \in \Out(F_{N})$ has  a simple (periodic) folding axis.

Pick a basepoint $x_{0} \in \mathcal{FF}_{N}$. For now, let us denote the projection map from the Outer space $CV_{N}$ to $\mathcal{FF}$ by $\Pi$. Then \cite[Proposition 8.1]{kapovich2022random} guarantees that: 

\begin{fact}\label{fact:kapo}
There exists $M_{0}>0$ such that the following holds. If $g$ is a fully irreducible and if the $d_{\mathcal{FF}}$-nearest point projections of $g^{-i}x_{0}$ and $g^{i} x_{0}$ onto $[\varphi^{-i} x_{0}, \varphi^{i} x_{0}]_{\mathcal{FF}}$ is at least $M_{0}$-apart, the first projection coming first, then $g$ is ageometric and triangular.
\end{fact}

The original \cite[Proposition 8.1]{kapovich2022random} is formulated in terms of $\operatorname{Pr}_{\gamma}$, but this can be replaced with the $d_{\mathcal{FF}}$-nearest point projection onto $\Pi(\gamma)$ by \cite[Lemma 4.2]{dowdall2017the-co-surface}. Furthermore, $[\varphi^{-i} x_{0}, \varphi^{i} x_{0}]_{\mathcal{FF}}$ uniformly fellow travels with subsegments $\Pi(\gamma_{i})$ of $\Pi(\gamma)$, where $\gamma_{i}$ exhausts $\gamma$ as $i$ tends to infinity. This justifies the reformulation.

Given Fact \ref{fact:kapo}, we take $M>M_{0}$ and run the proof of Theorem \ref{thm:mainWPD}: for each $\eta>0$ there exists $\epsilon>0$ such that $\mathcal{W}_{L, \epsilon}$ has asymptotic density $\ge 1-\eta$. For this $\epsilon$, elements of $\mathcal{W}_{L, \epsilon}(n)$ for large enough $n$ satisfy the assumption of Fact \ref{fact:kapo} by Lemma \ref{lem:WUseful}. Hence, $\mathcal{W}_{L, \epsilon}(n)$ consists of ageometric triangular fully irreducibles for large enough $n$.  By shrinking $\eta$, we conclude Theorem \ref{thm:mainGen}.
\end{proof}

%
%

\medskip
\bibliographystyle{alpha}
\bibliography{acyl}

\end{document}